\documentclass{ws-m3as}

\usepackage{color}
\usepackage{mlmath}
\usepackage{mathrsfs}
\usepackage[super]{cite}
\usepackage{url}
\usepackage{enumitem}

\setlist[enumerate,1]{
  label=(\roman*),    
  leftmargin=*, 
   labelsep=-0.25em,       
  align=left,      
  topsep=2pt,           
  partopsep=0pt,        
  itemsep=2pt,          
  parsep=0pt 
}

\newlength{\affilsep}
\begin{document}

\catchline{}{}{}{}{}

\title{Architecture Induces Structural Invariant Manifolds of Neural Network Training Dynamics}

\author{
    Jiajie Zhao$^{1,a}$, Tao Luo$^{1, 3,b,*}$ and 
    Yaoyu Zhang$^{1,2,c,*}$}

\address{
    $^1$ School of Mathematical Sciences, Institute of Natural Sciences and MOE-LSC, Shanghai Jiao Tong University, Shanghai, 200240, China.
    \\ [\affilsep]
    $^2$School of Artificial Intelligence, Shanghai Jiao Tong University, Shanghai, 200240, China
    \\[\affilsep]
    $^3$CMA-Shanghai, Shanghai Jiao Tong University, Shanghai, 200240, China.
    \\[\affilsep]
$^a$zjj0216@sjtu.edu.cn\\[\affilsep]$^b$luotao41@sjtu.edu.cn\\[\affilsep]$^c$zhyy.sjtu@sjtu.edu.cn\\[\affilsep]
$^*$Corresponding authors.
}

\maketitle

\begin{abstract}
While architecture is recognized as key to the performance of deep neural networks, its precise effect on training dynamics has been unclear due to the confounding influence of data and loss functions. This paper proposed an analytic framework based on the geometric control theory to characterize the dynamical properties intrinsic to a model's parameterization. We prove that the Structural Invariant Manifolds (SIMs) of an analytic model $F(\vtheta)(\vx)$—submanifolds that confine gradient flow trajectories independent of data and loss—are unions of orbits of the vector field family $\{\nabla_{\vtheta} F(\cdot)(\vx)\mid\vx\in\sR^d\}$. We then prove that a model's symmetry, e.g., permutation symmetry for neural networks, induces SIMs. Applying this, we characterize the hierarchy of symmetry-induced SIMs in fully-connected networks, where dynamics exhibit neuron condensation and equivalence to reduced-width networks. For two-layer networks, we prove all SIMs are symmetry-induced, closing the gap between known symmetries and all possible invariants. Overall, by establishing the framework for analyzing SIMs induced by architecture, our work paves the way for a deeper analysis of neural network training dynamics and generalization in the near future.

\end{abstract}

\keywords{neural network architecture; training dynamics; geometric control theory; structural invariant manifold}

\ccode{AMS Subject Classification: 68T07, 34H05, 93C10, 93B03, 93B27
}

\section{Introduction}

Neural networks serve as the core engine of modern AI applications. The architecture of a network—that is, its specific scheme for parameterizing functions—is widely recognized as the primary factor influencing its training behavior and ultimate generalization performance on a given task~\cite{krizhevsky2012imagenet,he2016deep,vaswani2017attention}. Nevertheless, the nonlinear nature of these architectures gives rise to highly nonlinear training dynamics, making the analysis of these dynamics and the precise consequences of architectural choices a persistently challenging problem\cite{ma2020towards}.

In recent years, several theoretical developments in deep learning have shed light on this problem. One line of research focuses on a key phenomenon in nonlinear training dynamics known as condensation\cite{luo2021phase,xu2025overview} (also referred to as quantization\cite{maennel2018gradient}, weight clustering\cite{brutzkus2019larger}, or alignment\cite{min2023early}). This widely observed process describes how neurons within a layer tend to align with one another during training. The study of condensation illuminates how neural networks adaptively extract features from data and reveals an implicit bias towards simpler functions that can be expressed by narrower networks\cite{xu2025overview}. Crucially, condensation results from the nonlinear network architecture and is absent in any linear models. In addition, a series of works have revealed that the permutation symmetry of neural network architectures profoundly impacts both training dynamics and the loss landscape's critical point distribution\cite{fukumizu2019semi,ziyin2023symmetry,simsek2021geometry,zhang2021embedding}.  Regarding the dynamics, it has been shown that permutation-invariant subspaces are also invariant under the training dynamics\cite{simsek2021geometry,ziyin2023symmetry}. Regarding the loss landscape, the embedding principle demonstrates that a network inherits all critical points from any narrower network within its architecture\cite{zhang2021embedding,simsek2021geometry,fukumizu2019semi}. Furthermore, some recent works have leveraged Lie brackets and the Frobenius theorem to systematically identify conserved quantities and the lower-dimensional invariant manifolds they induce in nonlinear models like deep linear and ReLU networks\cite{marcotte2023abide,marcotte2025intrinsic}. These conservation laws, inherent to the model architecture, constrain the models' global training dynamics.

Despite these advancements, uncovering the exact impact of a nonlinear architecture on training dynamics remains challenging, primarily due to the difficulty of isolating its effect from the complications of the training data and loss function. In this paper, we address this challenge by introducing the concept of structural invariant manifold (SIM), which is defined as a submanifold of the parameter space that confines gradient flow trajectories independent of training data and loss, as the central object for our study. By employing the geometric control theory, in particular the Hermann--Nagano Theorem \cite{herman-nagano}, we 
uncover the dynamical effect of architecture as follows:
\parbox{\textwidth}{\centering\textbf{Architecture partitions the parameter space into nonintersecting orbits.}}
These orbits and their unions give rise to all SIMs of the gradient flow. Note that all models possess a trivial SIM, i.e., the entire parameter space $\sR^M$. Our results yield a key insight into the dichotomy between linear and nonlinear models: a generic linear model possesses only the trivial SIM. Consequently, the existence of non-trivial SIMs, which often have much lower dimensions than the full parameter space, is a hallmark of how a nonlinear architecture fundamentally shapes training dynamics. We remark that our framework, grounded in geometric control theory, offers a unified mathematical foundation for the analysis of architecture-induced invariant structures. This approach bridges the gap between the separate treatments of invariant structures resulting from symmetry or conservation law in the literature\cite{simsek2021geometry,ziyin2023symmetry,marcotte2023abide,marcotte2025intrinsic}.

In general, uncovering the orbits of complex nonlinear models like neural networks is technically difficult. In this work, we identify a general family of architectural properties—namely, invariant maps and their induced symmetry groups and infinitesimal symmetries—that can conveniently reveal a series of SIMs. By determining all such symmetry groups and infinitesimal symmetries for general deep neural networks, we uncover a large family of SIMs with a hierarchical structure: non-trivial SIMs exist with dimensions ranging from low to high, where each lower-dimensional SIM exhibits neuron condensation and is functionally equivalent to a reduced-width network. While obtaining all SIMs for neural networks remains a general challenge, we take a step forward by proving that, for generic two-layer networks, all SIMs are indeed symmetry-induced. The properties and analytical techniques we develop for studying these SIMs are broadly applicable and extend to other nonlinear models, such as matrix factorization\cite{bai2024connectivity, koren2009matrix}, and to other architectures, including  Convolutional Neural Networks\cite{krizhevsky2012imagenet,zhang2022linear} and Transformers\cite{vaswani2017attention,chen2025from}. 

Overall, this paper establishes an analytic framework for identifying SIMs induced by model architecture, thereby elucidating how specific architectural designs inherently constrain gradient flow training dynamics globally. It is important to note, however, that the realized training dynamics and ultimate generalization performance are also profoundly influenced by other factors, including the target function, training data, initialization, and loss function. An important and promising direction for future research is to study how these elements interact with these architecture-induced geometric structures to determine the actual trajectory of training and generalization performance.

Our main results are summarized as follows:

\begin{itemize}
    \item \textbf{A Geometric Framework for SIMs (Theorem~\ref{def:sim}):} We introduce the concept of  \emph{Structural Invariant Manifolds (SIMs)}, and we prove they are unions of orbits of the vector field family $\fF = \left\{ \nabla_{\vtheta} F(\cdot)(\vx) \mid \vx \in \mathbb{R}^d \right\}$ (Section~\ref{sec:SIM}). This result establishes a foundational framework for analyzing the gradient dynamics constraints imposed by the network architecture.

    \item \textbf{Symmetry-Induced SIMs (Theorem~\ref{cor:symmetries induced sim}):} We prove that invariant maps can induce SIMs. This mechanism forges a crucial link between architectural symmetries and SIMs, as these invariant maps are themselves a natural consequence of a model's underlying symmetries.

    \item \textbf{Characterization in Deep Neural Networks (Theorem~\ref{thm:all symmetry of neural network}):} Building on this connection, we provide a  characterization of symmetry-induced SIMs in deep neural networks. This result is a crucial step towards understanding the complete set of SIMs in deep neural networks.

    \item \textbf{Complete Characterization for Generic Two-Layer Neural Networks (Theorem~\ref{thm:final orbit}):} We prove that for generic two-layer neural networks, all SIMs are induced by architectural symmetries. This provides a complete characterization of SIMs for this model class.
\end{itemize}

\section{Preliminary}

\label{sec:preliminary}
\subsection{Problem setting}
We define a \textbf{parametric model} as a map  
$
F:\sR^M \to C(\sR^d, \sR),
$  
where \(M\) and \(d\) are positive integers, and \(C(\sR^d, \sR)\) denotes the set of continuous functions from \(\sR^d\) to \(\sR\). 
Given a parameter  \(\vtheta \in \sR^M\), the  output function \(F(\vtheta)\) is  a function from \(\sR^d\) to \(\sR\).
The value of this function for an input $\vx\in \sR^d$ is denoted by $F(\vtheta)(\vx)$.
Given a parametric model $F$, a dataset $S=\{(\vx_i,y_i)\}_{i=1}^{n}$ and an analytic loss function $\ell: \sR\times \sR \to \sR$,  we can define the empirical loss function as
$L(\vtheta)=\sum_{i=1}^n \ell(F(\vtheta)(\vx_i),y_i)$.
We analyze in this work the gradient flow given by
\begin{equation}
    \frac{\rd\vtheta}{\rd t} = -\nabla_{\vtheta}L(\vtheta) = -\sum_{i=1}^{n}\nabla \ell(F(\vtheta)(\vx_i), y_i) \, \nabla_{\vtheta} F(\vtheta)(\vx_i).
    \label{eq:gradient flow}
\end{equation}
Throughout the paper, \( \nabla \ell \) denotes the gradient of the loss function \( \ell \) with respect to its first argument,  and we assume that \( F \) is an analytic parametric model (see Definition~\ref{def:analytic model}). When defining \( \nabla_{\vtheta} F(\vtheta)(\vx) \), we regard \( F \) as a function of two variables, \( F(\vtheta, \vx) \), and interpret \( \nabla_{\vtheta} F(\vtheta)(\vx) \) as the partial derivative of \( F \) with respect to \( \vtheta \), evaluated at the point \( (\vtheta, \vx) \).
 Additionally, while the objective function \( L(\vtheta) \) depends on both the dataset \( S \) and the loss function \( \ell \), we omit these dependencies from the notation for simplicity.

\begin{definition}[analytic parametric model]
    A parametric model \( F(\vtheta)( \vx) \), where \( \vtheta \in \mathbb{R}^M \) and \( \vx \in \mathbb{R}^d \), is called an \textbf{analytic parametric model} if \( F \), considered as a function of $\vtheta$ and $\vx$, is a real-valued analytic function.
    \label{def:analytic model}
\end{definition}

A major objective of this paper is to discuss invariant manifolds of Eq.~\eqref{eq:gradient flow} that are independent of loss function $\ell$ and dataset $S$.
The definition of an invariant manifold is provided in Definition~\ref{def:invariant manifold}.

\begin{definition}[vector field induced invariant set (manifold)\footnote{See Definition~\ref{def:immersed submanifold} for definition of an immersed submanifold, and Definition~\ref{def: analytic vector fields} for the definition of an analytic vector field.}]
Suppose $M$ is a positive integer and $\fM$ is a subset of $\sR^M$. Let $X$ be an analytic vector field on \( \sR^M\), and let \( \vtheta(t)\) denote the solution to the Cauchy problem \( \dot{\vtheta} = X(\vtheta),\vtheta(0) = \vtheta_0 \). 
We say that \( \mathcal{M} \) is an \textbf{invariant set} (with respect to the vector field \( X \)) if for every \( \vtheta_0 \in \mathcal{M} \), the solution \( \vtheta(t) \) remains in \( \mathcal{M} \) for all \( t \) in its maximal interval of existence.  We also say \textbf{$\fM$ is invariant under $X$}. Moreover, if $\fM$ is an immersed submanifold of $\sR^M$, we say $\fM$ is an \textbf{invariant manifold}.
\label{def:invariant manifold}
\end{definition}

In this paper, the primary parametric models we consider are the neural networks defined in Definitions~\ref{def:fully connected network} and~\ref{def:neural network}.

\begin{definition}[multi-layer fully-connected neural network]
Let $L\geq 2$ be an integer. 
Consider the neural network \( F(\vtheta)(\vx) \) defined inductively by
\[
\begin{aligned}
\va^{(0)} &= \vx, \\
\va^{(l)} &= \sigma\left(\vW^{(l)} \va^{(l-1)} + \vb^{(l)}\right), \quad &&\text{for } l = 1, 2, \ldots, L-1, \\
\va^{(L)} &= \vW^{(L)} \va^{(L-1)}+\vb^{(L)},
\end{aligned}
\]
where \( F(\vtheta)(\vx) = \va^{(L)} \) and the parameters are \( \vtheta = \left(\vW^{(l)}, \vb^{(l)}\right)_{l=1}^L \). Here, each \( \vW^{(l)} \) is an \( n_l \times n_{l-1} \) matrix, and \( \vb^{(l)} \), \( \va^{(l)} \) are vectors in \( \mathbb{R}^{n_l} \). The activation function $\sigma:\sR\to\sR$ is assumed to be real-analytic and acts entrywise on vectors. Technically when writing \( \vtheta = \left(\vW^{(l)}, \vb^{(l)}\right)_{l=1}^L \) the matrix $\vW^{(l)}$ should be flattened to a vector, but we omit it for simplicity. In this paper we consider scalar output, i.e. $n_L=1$.
\label{def:fully connected network}    
\end{definition}

\begin{definition}[two-layer neural network]
The network is represented as 
$
F(\vtheta)(\vx) = \sum_{i=1}^m a_i \sigma(\vw_i^{\T} \vx),
$
where \(\vx \in \mathbb{R}^d,a_i\in \sR,\vw_i\in \sR^d\), \(\vtheta = (a_i,\vw_i)_{i=1}^m \in \mathbb{R}^{(d+1)m}\), and \(m\) is the width of the network. The function \(\sigma: \mathbb{R} \rightarrow
\mathbb{R}\) is the activation function, which is assumed to be a non-polynomial, real analytic function.  
\label{def:neural network}
\end{definition}

Numerous symmetries exist in neural networks. We will introduce the concept here and discuss it in detail in Section~\ref{sec:symmetry}.
\begin{definition}[symmetry group]
Let \( F \) be an analytic parametric model, and let \( G \) be a group (or semigroup) acting on the parameter space of \( F \). If the action of any element \( g \in G \) leaves the output of \( F \) invariant for all inputs, i.e.,
$
F(g(\boldsymbol{\theta}))(\mathbf{x}) = F(\boldsymbol{\theta})(\mathbf{x}),  \forall g \in G,  \forall \boldsymbol{\theta} \in \mathbb{R}^M,  \forall \mathbf{x} \in \mathbb{R}^d,
$
then \( G \), together with its action, is called a \textbf{symmetry group (or semigroup)} of \( F \).
Furthermore, if every action of \( G \) is an orthogonal linear transformation, then \( G \) is called an \textbf{orthogonal symmetry group}.
\label{def:symmetry group_pre}
\end{definition}

\subsection{Orbit}
To analyze the dynamics of $\vtheta(t)$, we now introduce concepts from geometric control theory. A detailed introduction of geometric control theory is provided in~\ref{appendix:def}.

\begin{definition}[orbit, page 33 of Ref.~\refcite{jurdjevic1997geometric}]\footnote{Please see~\ref{appendix:def}
for the definition of $e^{tX}$, analytic manifold,  analytic vector field and pseudogroup.}
    Let $\fF$ be a family of analytic vector fields on an analytic manifold $\fM$. Let $G=G(\fF)$ be the group (pseudogroup) of diffeomorphisms (local diffeomorphisms) generated by $\{e^{tX}\mid t\in\sR, X\in\fF\}$ under composition. For any $\vtheta\in \fM$, we define the \textbf{orbit} of $\fF$ through $\vtheta$ as $\{g(\vtheta)\mid g\in G\}$, which we denote by $O_{\fF}(\vtheta)$. 
\end{definition}

Given a family of analytic vector fields, each of its orbits forms an analytic immersed submanifold. The dimension of an orbit is determined by the Lie closure of the vector field family, as stated in Definition~\ref{def:Lie closure} and Theorem~\ref{thm:Hermann-Nagano}. 

In Ref.~\refcite{jurdjevic1997geometric}, Theorem~\ref{thm:Hermann-Nagano} is stated under an analytic regularity assumption, whereas Corollary~\ref{theorem: page 44 book} is presented as a theorem under smooth regularity. Given that this paper operates within analytic regularity, we introduce an analytic version of Corollary~\ref{theorem: page 44 book}. This version is a direct consequence of Theorem~\ref{thm:Hermann-Nagano}, and we therefore designate it as a corollary.

\begin{definition}[Lie closure]
Let $\fM$ be an analytic manifold and  $\fF$ be a family of analytic vector fields on  $\fM$. We use $\mathrm{Lie}(\fF)$ to denote the Lie algebra of analytic
vector fields generated by $\fF$. For any point $\vtheta\in \fM$, $\mathrm{Lie}_{\vtheta} (\fF)$ is defined to be the set of all tangent vectors $V(\vtheta)$ with $V$ in $\mathrm{Lie}(\fF)$. We call $\mathrm{Lie}_{\vtheta} (\fF)$ the \textbf{Lie closure} of $\fF$ at $\vtheta$.
\label{def:Lie closure}
\end{definition}

\begin{theorem}
(Hermann--Nagano Theorem, Theorem 6 in Section 2 of Ref.~\refcite{jurdjevic1997geometric}) Let $\fM$ be an analytic manifold, and $\mathcal{F}$ a family of analytic vector fields on $\fM$. Then:
\begin{enumerate}
    \item Each orbit of $\mathcal{F}$ is an (immersed) analytic submanifold of $\fM$.
    \item If $\fN$ is an orbit of $\mathcal{F}$, then the tangent space of $\fN$ at $\vtheta$ is given by $\operatorname{Lie}_{\vtheta}(\mathcal{F})$. In particular, the dimension of $\operatorname{Lie}_{\vtheta}(\mathcal{F})$ is constant as $\vtheta$ varies over $\fN$.
\end{enumerate}
\label{thm:Hermann-Nagano}
\end{theorem}

\begin{corollary}(Theorem 3 in Section 2 of Ref.~\refcite{jurdjevic1997geometric})
Let $\fM$ be an analytic manifold and $\fF$ be a family of analytic vector fields on $\fM$.
    Suppose that $\mathcal{F}$ is such that $\mathrm{Lie}_{\vtheta_0}(\mathcal{F})=T_{\vtheta_0}\fM$ for some $\vtheta_0$
    in $\fM$. Then the orbit  of $\mathcal{F}$ through $\vtheta_0$ is open. If, in addition, $\mathrm{Lie}_{\vtheta}(\mathcal{F})=T_{\vtheta}\fM$ for each $\vtheta$ in $\fM$, and if $\fM$ is connected, then there is only one orbit of $\mathcal{F}$ equal to $\fM$.
   \label{theorem: page 44 book} 
\end{corollary}

\section{Structural Invariant Manifold (SIM) and its framework }
\label{sec:SIM}

\begin{figure}[h]
    \centering
    \includegraphics[width=0.95\linewidth]{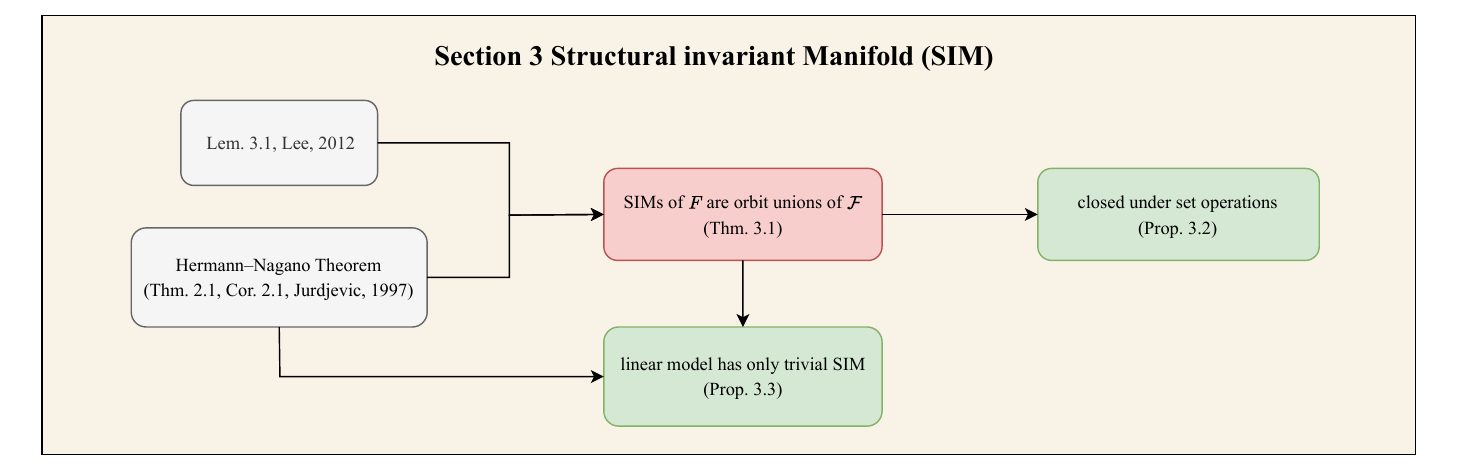}
    \caption{Flowchart illustrating the logical structure of Section~\ref{sec:SIM}.}
    
    \label{fig:sec_3}
\end{figure}

This section introduces the concept of Structural Invariant Manifolds (SIMs) and develops the corresponding theoretical framework. The logical flow of the section is illustrated in Figure~\ref{fig:sec_3}. Our main result, Theorem~\ref{def:sim}, establishes that SIMs are precisely the unions of orbits  of the vector field family \(\mathcal{F}\). We then explore additional properties of SIMs, showing that they are closed under standard set operations (Proposition~\ref{lemma:closed under set operations}) and that linear models admit only trivial SIMs (Proposition~\ref{example:linear}). This rigorous foundation sets the stage for the deeper analysis of SIMs in the remainder of the paper.

\subsection{SIM as a key tool for the recovery puzzle}
Considering the simplest setup where a parametric model $F(\vtheta)(\vx)$ is used to recover a target function $f^*\in\{F(\vtheta)(\cdot)\mid \vtheta \in \sR^M\}$ from $n$ training samples $\{(\vx_i,f^*(\vx_i))\}_{i=1}^{n}$. The gradient flow training dynamics is written as
\begin{equation}
    \frac{\rd\vtheta}{\rd t}=-\nabla_{\vtheta} L(\vtheta)=-\sum_{i=1}^{n}\nabla\ell(F(\vtheta)(\vx_i),f^*(\vx_i))\nabla_{\vtheta} F(\vtheta)(\vx_i).
\label{eq:gradient flow_GC}
\end{equation}
A fundamental question in machine learning, known as the recovery problem, is to identify the conditions that allow above dynamics to successfully find the target function $f^*$. If $F(\vtheta)(\vx)$ is a linear model in $\vtheta$ with linearly independent basis and a proper loss $\ell(\cdot,\cdot)$, then it is well-known that $f^*$ can be recovered generically from $n\geq M$ samples\cite{zhang2022linear}. However, when we change $F(\vtheta)(\vx)$ to a nonlinear model,  our understanding becomes extremely limited. 

A particularly mysterious phenomenon is that nonlinear models like neural networks can recover certain targets even under severe overparameterization $n\ll M$\cite{zhao2024disentangle,zhang2022linear,zhang2023optimistic,zhang2023geometry}. This phenomenon sparks the following recovery puzzle:
\parbox{\textwidth}{\centering\textbf{How neural networks recover targets under overparameterization?}}
Note that this puzzle is a specialization of the widely acknowledged generalization puzzle in deep learning theory---why overparameterized neural networks often generalize well\cite{breiman2018reflections,zhang2016understanding}. Yet, we argue that the recovery puzzle well serves as the cornerstone of the generalization puzzle: (i) the notion of recovery resolves the ambiguity in the notion of ``generalize well''; (ii) understanding the conditions for recovery is often the first and the key step towards understanding generalization as in the cases of linear regression, signal processing\cite{shannon1948mathematical,luke1999origins} (e.g., Nyquist-Shannon sampling theorem\cite{shannon1948mathematical}), and compressed sensing\cite{candes2006robust,donoho2006compressed,candes2008introduction}. 

In recent years, progress has been made by solving some weaker versions of the recovery puzzle. Ref.~\refcite{zhang2022linear} proves a recovery guarantee for neural networks under overparameterization in the sense of local linear recovery, i.e., recovering targets in the tangent space of some optimal point in the target set $F^{-1}(f^*)$. Ref.~\refcite{zhang2023geometry} makes a step further to prove a recovery guarantee in the sense of local recovery, i.e., recovering targets in the neighbourhood of the target set $F^{-1}(f^*)$. Despite the progress, how one can leverage these local recovery guarantees to a global one remains an extremely difficult problem. Particularly, we lack means to globally back-trace the gradient flow dynamics from the vicinity of the target set to see if there exists a generic initialization that reliably access $F^{-1}(f^*)$ for $n<M$ training samples.  

Motivated by recent results that demonstrate existence of lower dimensional ($<M$) invariant subspaces independent to training data and loss induced by the permutation symmetry of neural networks architecture\cite{simsek2021geometry,ziyin2023symmetry}, we realize that these architecture-induced invariant manifolds could serve as the key tool for the global tracing of gradient flow dynamics. For the convenience of study, we first provide a formal definition as follows. 

\begin{definition}[structural invariant manifold (SIM)]
Let $F(\vtheta)(\vx),\vtheta\in \sR^M, \vx\in \sR^d$
be an analytic parametric model. For a subset $\fM\subset \sR^M$, we say $\fM$ is a \textbf{structural invariant set} if it is invariant under $-\nabla_{\vtheta}L(\vtheta)$ in Eq.~\eqref{eq:gradient flow} for any real analytic loss function $\ell:\sR\times \sR \to \sR$ and dataset $S$. Moreover, if $\fM$ is an immersed submanifold of $\sR^M$, we say $\fM$ is a \textbf{structural invariant manifold.}\footnote{ By convention, the empty set is neither a structural invariant set nor a structural invariant manifold.}
  \label{def: structural invariant manifold}  
\end{definition}

The concept of SIM is introduced to capture the intrinsic dynamical consequence of a model's architecture, independent of any particular dataset or loss function. Example~\ref{example: exp network} illustrates a nontrivial SIM that arises in a simple nonlinear model. As we will show in the next section, this manifold emerges as a consequence of the model’s permutation symmetry.

\begin{example}[SIMs of two-neuron exponential neural network]
Consider a two-neuron neural network with exponential activation and a one-dimensional input \(x\). The model is given by:
\[
F(\vtheta)(x) = a_1e^{w_1 x} + a_2e^{w_2 x},
\]
where \(\vtheta = (a_1,w_1,a_2,w_2) \in \mathbb{R}^4,x\in \sR\). It is easy to verify that $\fM=\{(a_1,w_1,a_2,w_2)\in \sR^4\mid a_1=a_2,w_1=w_2\}$ is a SIM.
Later in Theorem~\ref{thm:final orbit} we will see that all SIMs of $F$ are $\sR^4$,
$\fM$ and $\sR^4\setminus\fM$.

\label{example: exp network}
\end{example}
The justification for the utility of lower dimensional SIMs in global trajectory tracing is as follows: \\ 
(i) \textbf{Enabling recovery under overparameterization:} a $d<M$ dimensional SIM enables us to study the gradient flow on this confined lower dimensional manifold.
If the target set $F^{-1}(f^*)$ intersecting with certain $d<M$ dimensional SIM, then target recovery by intuition is possible with $d\leq n < M$ training samples.\\
(ii) \textbf{Enabling strong complexity control:} a $d<M$ dimensional SIM makes it possible for the model to keep the complexity (marked by the effective degrees of freedom) $\leq d$ for an arbitrarily long time: (1) on the SIM the complexity is constrained for infinite time; (2) the closer some $\vtheta(t)$ is to the SIM, the longer afterwards the output complexity is upper bounded approximately by $d$.

SIMs emerge directly from a model's architecture, providing the key utilities for analyzing global dynamics described above. Therefore, this work makes an effort to establish a theoretical foundation—using geometric control theory—for the systematic identification of all SIMs in analytic parametric models with a focus on the neural network architecture.

\subsection{SIMs as orbit unions of \(\fF\)}
A central challenge in the study of SIMs for Eq.~\eqref{eq:gradient flow} lies in isolating the influence of model architecture from the confounding effects of training data and loss function. In this work, we propose a relaxation of the dynamics that resembles a geometric control problem, revealing a connection between the SIMs of the model \(F\) and the orbits of the induced vector fields \(\fF\), defined as follows.

\begin{definition}[induced vector fields]
 Let $F(\vtheta)(\vx)$ be an analytic parametric model with $\vtheta\in \sR^M$ and $\vx \in \sR^d$. Define the family of vector fields
 $$
\fF = \left\{ \nabla_{\vtheta} F(\cdot)(\vx) \mid \vx \in \mathbb{R}^d \right\}.
$$
$\fF$
is called the \textbf{induced vector fields of the model}. 

 \label{def:orbit of model's gradient}
\end{definition}

Our relaxation of the dynamics in Eq.~\eqref{eq:gradient flow} proceeds as follows. For each parameter vector \(\vtheta \in \mathbb{R}^M\), we observe that the gradient \(-\nabla_{\vtheta} L(\vtheta)\) lies within the span of the model's gradients, i.e.,  
\[
-\nabla_{\vtheta} L(\vtheta) \in \mathrm{span}\left(\{\nabla_{\vtheta} F(\vtheta)(\vx_i)\}_{i=1}^{n}\right) \subseteq \mathrm{span}\left(\{\nabla_{\vtheta} F(\vtheta)(\vx) \mid \vx \in \mathbb{R}^d\}\right).
\]
This observation implies that the gradient flow trajectories are encapsulated in the orbits of \(\{\nabla_{\vtheta} F(\cdot)(\vx_i)\}_{i=1}^{n}\), hence in the orbits of \(\fF\), which is determined solely by the model architecture.

With a detailed theoretical derivation below, we arrive at the first key result of our work in Theorem~\ref{def:sim}, which ensures that the orbits of $\fF$ and their unions give rise to all SIMs. This theorem serves as the foundation for all our later analysis as it translates the seemingly complicated task of identifying all SIMs into a clean one: computing the orbits of \(\fF\). Building on this result, we can further explore several key questions: What SIMs arise under different neural network architectures? How do these manifolds emerge?

The proof of Theorem~\ref{def:sim} relies on Lemma~\ref{lemma:sim tangent} as an auxiliary result.  We therefore begin by stating and proving Lemma~\ref{lemma:sim tangent}, and then proceed to the proof of Theorem~\ref{def:sim}.

\begin{lemma}[Problem 9-2 of Ref.~\refcite{Lee2013ISM}]
Let \( X \) be a smooth vector field on \( \mathbb{R}^M \), and consider the Cauchy problem
\begin{equation}
\frac{\rd\vtheta}{\rd t} = X(\vtheta), \quad \vtheta(0) = \vtheta_0,
\label{eq:cauchy problem}
\end{equation}
with solution denoted by \( \vtheta(t) \). Suppose \( \mathcal{M} \subset \mathbb{R}^M \) is an immersed submanifold such that \( X(\vtheta) \in T_{\vtheta}\mathcal{M} \) for all \( \vtheta \in \mathcal{M} \). Then for any initial condition \( \vtheta_0 \in \mathcal{M} \), there exists \( \delta > 0 \) such that \( \vtheta(t) \in \mathcal{M} \) for all \( |t| < \delta \). Moreover, if \( \mathcal{M} \) is closed in \( \mathbb{R}^M \), then \( \vtheta(t) \in \mathcal{M} \) for all \( t \) in the maximal interval of existence.
\label{lemma:sim tangent}
\end{lemma}

\begin{proof}
We provide its proof for completeness. We prove local invariance first, then extend to global invariance under the closedness assumption.\\
\textbf{Local invariance.} Let \( \vtheta_0 \in \mathcal{M} \), and let \( \vtheta(t) \) be the solution to Eq.~\eqref{eq:cauchy problem}. Pick any \( t_0\in \sR \) such that \( \vtheta(t_0) \in \mathcal{M} \). Since \( \mathcal{M} \) is an immersed submanifold, there exists a local parameterization \( \psi: U \subset \mathbb{R}^k \to \mathbb{R}^M \) such that \( \psi(U) \subset \mathcal{M} \), \( \vtheta(t_0) \in \psi(U) \), and \( D\psi(u) \) has full column rank for all \( u \in U \). Let \( \vu_0 := \psi^{-1}(\vtheta(t_0)) \in U \), and consider the ODE in \( \mathbb{R}^k \):
\[
\frac{\rd\vu}{\rd t} = (D\psi^\T D\psi)^{-1} D\psi^\T X(\psi(\vu(t))), \quad \vu(t_0) = \vu_0.
\]
This defines a smooth vector field in \( \mathbb{R}^k \), hence admits a unique solution \( \vu(t) \) near \( t_0 \). Define \( \tilde{\vtheta}(t) := \psi(\vu(t)) \). By the chain rule,
\[
\frac{\rd\tilde{\vtheta}}{\rd t} = D\psi(\vu(t)) \cdot \frac{\rd\vu}{\rd t} = D\psi (D\psi^\T D\psi)^{-1} D\psi^\T X(\psi(\vu(t))).
\]
Since \( X(\psi(\vu(t))) \in T_{\psi(\vu(t))} \mathcal{M} \), $X(\psi(\vu(t))$ is in the image of $D\psi$. Thus, \( \frac{\rd\tilde{\vtheta}}{\rd t} = X(\tilde{\vtheta}(t)) \). Therefore, \( \tilde{\vtheta}(t) \) satisfies the same ODE as \( \vtheta(t) \) and coincides with it at \( t = t_0 \). By uniqueness, \( \vtheta(t) = \tilde{\vtheta}(t) \in \mathcal{M} \) near \( t_0 \). Taking \( t_0 = 0 \), we obtain \( \delta > 0 \) such that \( \vtheta(t) \in \mathcal{M} \) for all \( |t| < \delta \).\\
\textbf{Global invariance (if \( \mathcal{M} \) is closed).} Pick arbitrary $T$ in the maximal interval of existence. Without loss of generality we assume $T>0$. Define $A=\{t\in [0,T]\mid \vtheta(t)\in \fM\}$. By the local invariance, $A$ is open in $[0,T]$.
Since \( \vtheta(t) \) is continuous and \( \mathcal{M} \) is closed, \( A \) is also closed in \( [0, T] \). Since $0\in A$, \( A = [0, T] \) by connectedness. Therefore $\vtheta(T)\in \fM$.
Since \( T \) was arbitrary within the maximal interval of existence, it follows that \( \vtheta(t) \in \mathcal{M} \) for all  \( t \) in the maximal interval of existence.
\end{proof}

\begin{theorem}[SIMs of $F$ are orbit unions of $\fF$]
 Let \( F(\vtheta)(\vx),\vtheta\in \sR^M, \vx\in \sR^d \) be an analytic parametric model. Let \( \fF = \left\{ \nabla_{\vtheta} F(\cdot)(\vx) \mid \vx \in \mathbb{R}^d \right\} \).  
Let \( \fM \neq \emptyset \) be a subset (or immersed submanifold) of $\sR^M$. Then $\fM$ is a structural invariant set (or SIM) if and only if \( \fM \) is invariant under every vector field in \( \fF \), equivalently, $\fM$ is union of orbits of $\fF$.

\label{def:sim}
\end{theorem}

\begin{proof}
Let $\fM \subseteq \sR^M$. Consider the following statements:
\textbf{(i)} $\fM$ is a structural invariant set.
\textbf{(ii)} $\fM$ is invariant under every vector field in $\fF$.
\textbf{(iii)} $\fM$ is a union of orbits of $\fF$.
We will show that statements \textbf{(i)}\textbf{(ii)}\textbf{(iii)} are equivalent.

\medskip
\noindent
\textbf{(i) \(\Longrightarrow\) (ii)}:  
Assume \( \fM \) is invariant under the gradient flow \( -\nabla_{\vtheta} L(\vtheta) \) for any real analytic loss function and dataset.  
Let \( \vx \in \mathbb{R}^d \) be arbitrary. Consider the loss function \( \ell(s, t) = -s \) and the dataset \( S = \{(\vx, y)\} \) for some \( y \in \mathbb{R} \).  
Then  \( L(\vtheta) = -F(\vtheta)(\vx) \), and hence
$
-\nabla_{\vtheta} L(\vtheta) = \nabla_{\vtheta} F(\vtheta)(\vx).
$
So $\nabla_{\vtheta} L(\cdot)$ is a vector field in \( \fF \). Since \( \vx \in \mathbb{R}^d \) is arbitrary, \( \fM \) is invariant under every vector field in \( \fF \).

\medskip
\noindent
\textbf{(ii) \(\Longrightarrow\) (iii)}:  
Assume that \( \fM \subset \mathbb{R}^M \) is invariant under every vector field in \( \fF \). Since this invariance is preserved under the composition of flows, it follows that \( \fM \) is also invariant under every (local) diffeomorphism in the (pseudo) group generated by \( \fF \). Consequently, \( O_{\fF}(\vtheta)\subset \fM\) for any $\vtheta\in \fM$. Also,  
$\fM\neq \emptyset$. Therefore
 \( \fM \) is a union of orbits of \( \fF \).

\medskip
\noindent
\textbf{(iii) \(\Longrightarrow\) (i)}: We begin by presenting a lemma along with its proof.
\begin{lemma}
    Let $\fF$ be an arbitrary  family of  analytic vector fields on $\sR^M$. 
    Assume $X(\vtheta)$ is an analytic vector field that    satisfies the condition $X(\vtheta)\in \mathrm{Lie}_{\vtheta}(\fF),\forall \vtheta\in \sR^M$.  Then each orbit of $\fF$ is invariant under $X(\vtheta)$.
\label{lemma:orbit invariance}
\end{lemma}
\noindent\textbf{Proof for Lemma~\ref{lemma:orbit invariance}:}
Let $\fM$ be an orbit of $\fF$.
Fix any $\vtheta_0\in \fM$, and let $\vtheta(t),t\in I$ be the solution to the Cauchy problem $\frac{\rd \vtheta}{\rd t}=X(\vtheta),\vtheta(0)=\vtheta_0$, where $I$ is the maximal interval of existence.
We now prove that $\vtheta(t)\in \fM,\forall t\in I$.
Suppose for  contradiction that $\{t\in [0,+\infty)\cap I\mid \vtheta(t)\notin \fM\} \neq \emptyset$. Let $t_1=\inf \{t\in [0,+\infty)\cap I\mid \vtheta(t)\notin \fM\}$.
By Theorem~\ref{thm:Hermann-Nagano}, $\fM$ is an immersed submanifold, and its tangent space is  given by $T_{\vtheta}\fM=\mathrm{Lie}_{\vtheta}(\fF)$. Since $X(\vtheta)\in \mathrm{Lie}_{\vtheta}(\fF)$ for any $\vtheta\in \sR^M$, we have $X(\vtheta)\in T_{\vtheta}\fM,\forall \vtheta\in \fM $. 
 By Lemma~\ref{lemma:sim tangent}, there exists \( \delta > 0 \) such that \( \vtheta(t) \in \fM \) for all \( t \in [0, \delta) \), which implies \( t_1 > 0 \).

Applying Lemma~\ref{lemma:sim tangent} again at $\vtheta(t_1)$, we obtain a $\delta_1\in(0,t_1)$ such that
\( \vtheta(t) \in O_{\mathcal{F}}(\vtheta(t_1)), \forall t\in [t_1-\delta_1,t_1+\delta_1]\).
Since $\vtheta(t_1-\delta_1)\in \fM\cap O_{\fF}(\vtheta(t_1))$, 
$\fM=O_{\fF}(\vtheta(t_1))$. Therefore for all $0<t\leq t_1+\delta_1$, $\vtheta(t)\in\fM $, which contradicts that $t_1$ is the infimum.
Hence, the set \( \{t\in [0,+\infty)\cap I\mid \vtheta(t)\notin \fM\} \) is empty. Similarly, we have \( \{ t \in (-\infty,0]\cap I \mid \vtheta(t) \notin \fM \} = \emptyset \).  
So \( \vtheta(t)\in \fM \)  for all $t\in I$. Therefore $\fM$ is invariant under $X(\vtheta)$.
\hfill $\square$

We now return to the proof of \textbf{(iii) $\Longrightarrow$ (i)}.
Assume that \( \fM \) is a union of orbits of $\fF = \left\{ \nabla_{\vtheta} F(\cdot)(\vx) \mid \vx \in \mathbb{R}^d \right\}$. Consider an arbitrary dataset $S$ and loss function $\ell$. For any $\vtheta\in \sR^M$, the vector $-\nabla_{\vtheta}L(\vtheta)$ in Eq.~\eqref{eq:gradient flow} is a linear combination of vectors in $\fF|_{\vtheta}$, where  $\fF|_{\vtheta}$ is the evaluation of $\fF$ at $\vtheta$. Thus, for any $\vtheta\in \sR^M$, we have $-\nabla_{\vtheta}L(\vtheta)\in \mathrm{span} (\fF|_{\vtheta})\subset\mathrm{Lie}_{\vtheta}(\fF)$. By Lemma~\ref{lemma:orbit invariance}, each orbit of $\fF$ is invariant under the vector field $-\nabla_{\vtheta}L(\cdot)$. Hence, each orbit of $\fF$ is a SIM. Since \(\fM\) is a union of orbits of \(\fF\), it follows readily from the definition that \(\fM\) is a SIM.

Thus the three statements are equivalent. Furthermore, if $\fM$ is assumed to be an immersed submanifold of $\sR^M$, statement \textbf{(i)} may be replaced by the assertion that $\fM$ is a SIM.

\end{proof}

In defining SIM, we require it to be invariant to the gradient flow under any loss function \(\ell\) and dataset \(S\). However, Proposition~\ref{prop:data-independent} shows that, under mild assumptions on the loss function \(\ell_0\), an immersed submanifold is a SIM if and only if it is invariant to the gradient flow under this loss function $\ell_0$ and  any dataset 
$S$. Intuitively, data-independent invariance is strong enough to induce structural invariance.

\begin{proposition}
Let \(F(\vtheta)( \vx)\) be an analytic parametric model with \(\vtheta \in \sR^M\), and let \(\fM\) be an immersed submanifold of \(\sR^M\).  
Suppose the loss function \(\ell_0(s,t)\) is real analytic and satisfies:  
\(\forall s \in \sR, \exists t \in \sR\) such that \(\nabla \ell_0(s,t) \ne 0\). 
Then $\fM$ is a SIM if and only if $\fM$ is invariant under  $-\nabla_{\vtheta}L(\vtheta)$ in Eq.~\eqref{eq:gradient flow} for this  loss function $\ell_0$ and any dataset $S$.  
\label{prop:data-independent}
\end{proposition}

\begin{proof}
By definition of SIM, one direction is trivial. To prove the other direction, 
assume that $\fM$ is invariant under $-\nabla_{\vtheta}L(\vtheta)$ in Eq.~\eqref{eq:gradient flow} for any dataset $S$ and the  loss function $\ell_0$.

Let $\vx_0\in \sR^d$ be arbitrary, and let $S_0=\{(\vx_0,y)\}$ for some $y\in \sR$.  Denote $X(\vtheta)=\nabla_{\vtheta} F(\vtheta)(\vx_0)$.
 Then under the dataset $S_0$ and the loss function $\ell_0$, we have $-\nabla_{\vtheta}L(\vtheta)=-\nabla l_0(F(\vtheta)(\vx_0),y)X(\vtheta)$. By our initial assumption, $\fM$ is invariant under $-\nabla l_0(F(\vtheta)(\vx_0),y)X(\vtheta)$ for any $y\in \sR$. Define $\fF'=\{-\nabla l_0(F(\vtheta)(\vx_0),y)X(\vtheta)\mid y\in \sR\}$. Since $\fM$ is invariant under any vector field in $\fF'$, it follows that  $\fM$  is invariant under compositions of flows generated by vector fields in  $\fF'$. So $O_{\fF'}(\vtheta)\subset \fM,\forall\vtheta\in \fM$. Now, fix any $\vtheta_0\in \fM$.  Then $O_{\fF'}(\vtheta_0)\subset \fM$. Let $\vtheta(t),t\in I$ denote the solution of the Cauchy problem $\frac{\rd \vtheta}{\rd t}=X(\vtheta),\vtheta(0)=\vtheta_0$, where $I$ is the maximal interval of existence. We now prove that $\vtheta(t)\in \fM,\forall t\in I$. Since $O_{\fF'}(\vtheta_0)\subset \fM$, it is sufficient to prove that $\vtheta(t)\in O_{\fF'}(\vtheta_0),\forall t\in I$. By assumption of $\ell_0$, for any $\vtheta\in \sR^M$,
there exists $y_0\in \sR$ such that $\nabla \ell(F(\vtheta)(\vx_0),y_0))\neq 0$. Therefore, for any $\vtheta\in \sR^M $, we have $X(\vtheta)\in \fF'|_{\vtheta}\subset \mathrm{Lie}_{\vtheta}(\fF')$. Applying Lemma~\ref{lemma:orbit invariance}, it follows that $O_{\fF'}(\vtheta_0)$ is invariant under $X(\vtheta)$. So $\vtheta(t)\in O_{\fF'}(\vtheta_0)
\subset \fM,\forall t\in I $. Thus, $\fM$ is invariant under $X(\vtheta)$. Since our choice of $\vx_0$ is arbitrary, $\fM$ is invariant under any vector field in $\fF = \left\{ \nabla_{\vtheta} F(\cdot)(\vx) \mid \vx \in \mathbb{R}^d \right\}$. It follows from Theorem~\ref{def:sim} that $\fM$ is a SIM.
\end{proof}

 Structural invariant sets are closed under set operations, as shown in Proposition~\ref{lemma:closed under set operations}.

\begin{proposition}[structural invariant sets are closed under set operations]
Let \( F(\vtheta)(\vx) \) be an analytic parametric model with parameter \( \vtheta \in \mathbb{R}^M \), and let \( \fE \) denote the collection of all structural invariant sets of the model $F$, augmented by the empty set. Then the following properties hold:
\begin{enumerate}

\item If \( \fM \in \fE \), then its complement \( \mathbb{R}^M \setminus \fM \in \fE \).
\item If \( \{\fM_i\}_{i \in I} \subseteq \fE \) is any  collection of structural invariant sets indexed by \( I \), then both the intersection  \( \bigcap_{i \in I} \fM_I \) and the union \( \bigcup_{i \in I} \fM_i \)  belong to \( \fE \).
\end{enumerate}
\label{lemma:closed under set operations}
\end{proposition}

\begin{proof}
   By definition, $\sR^M=\bigcup_{j\in J} O_j$, where $J$ is an index set,   $O_j,j\in J$ are all orbits of $\fF = \{\nabla_{\vtheta}F(\cdot)(\vx)\mid \vx\in \sR^d\}$. Besides, $O_s\cap O_t=\emptyset$ if $s,t\in J$ and $s\neq t$. 
   By Theorem~\ref{def:sim}, elements of $\fE$ is of the form $\fM=\bigcup_{j\in J'} O_j$, where $J'\subset J$ is an arbitrary index set ($J'$ may be empty). 
\\   
\textbf{(i) Complement:} Let $\fM=\bigcup_{j\in J'} O_j$ be an arbitrary structural invariant set, where $J'\subset J$ is an index set. Then
$\sR^M\setminus\fM=\bigcup_{j\in J\setminus J'}O_j$. Therefore $\sR^M\setminus\fM \in \fE$.
\\
\textbf{(ii) Union and Intersection:} Let $\fM_i=\bigcup_{j\in J_i} O_j$ be arbitrary structural invariant sets for  $i\in I$, where $I$ is an arbitrary index set, and $J_i\subset J$ for all $i\in I$.
Then it is straightforward to verify $\bigcup_{i\in I} \fM_i=\bigcup_{j\in\cup_{i\in I} J_i} O_j$, and $\bigcap_{i\in I} \fM_i=\bigcup_{j\in\cap_{i\in I} J_i} O_j$.
Therefore both $\bigcup_{i\in I} \fM_i$ and $\bigcap_{i\in I} \fM_i$ are in $\fE$.
\end{proof}

As demonstrated in Proposition~\ref{example:linear}, linear models possess only trivial SIM $\sR^M$, highlighting that the presence of nontrivial SIMs is a distinctive characteristic of nonlinear systems.
\begin{proposition}[linear model has only trivial SIM]
Let
$\{\psi_1( \vx),\ldots,$ $\psi_M( \vx)\}$   be a set of linearly independent analytic functions defined on $\sR^d$. 
Consider the linear model \(F(\vtheta)(\vx) = \sum_{i=1}^M \theta_i \psi_i( \vx)\), where \(\vtheta=(\theta_1,\ldots,\theta_M) \in \mathbb{R}^M,\vx\in \sR^d\). Then $F$ has only the trivial SIM,  $\sR^M$.
\label{example:linear}
\end{proposition}

\begin{proof}
By Theorem~\ref{def:sim}, a SIM is a union of orbits of the family \( \fF = \{ (\psi_1(\vx), \ldots, \psi_M(\vx)) \mid \vx \in \sR^d \} \). Therefore, it suffices to show that \( \fF \) has a single orbit equal to \( \sR^M \).

 Fix any $\vtheta_0\in \sR^M$. Let $\fF|_{\vtheta_0}$ denote the evaluation of $\fF$ at some $\vtheta_0$. 
Suppose, for contradiction that  \( \mathrm{span}(\fF|_{\vtheta_0}) \subsetneq \sR^M \). Then there exists a nonzero vector \( \vc = (c_1, \ldots, c_M) \in \sR^M \) such that \( \vc \) is orthogonal to \( \mathrm{span}(\fF|_{\vtheta_0}) \), i.e., \( \sum_{i=1}^M c_i \psi_i(\vx) = 0,\forall \vx\in \sR^d \). This contradicts the assumption that \( \{ \psi_1(\vx), \ldots, \psi_M(\vx) \} \) is a linearly independent set of functions. Therefore, \( \mathrm{span}(\fF|_{\vtheta}) = \sR^M \) for all \( \vtheta \in \sR^M \).
Since \( \mathrm{span}(\fF|_{\vtheta}) \subset \mathrm{Lie}_{\vtheta}(\fF) \subset \sR^M \), it follows that \( \mathrm{Lie}_{\vtheta}(\fF) = \sR^M \) for all \( \vtheta \in \sR^M \). By Corollary~\ref{theorem: page 44 book}, this implies that \( \fF \) has a single orbit equal to \( \sR^M \).
\end{proof}

\section{Symmetry and Symmetry-Induced SIM}
\label{sec:symmetry}

\begin{figure}[h]
    \centering
    \includegraphics[width=0.95\linewidth]{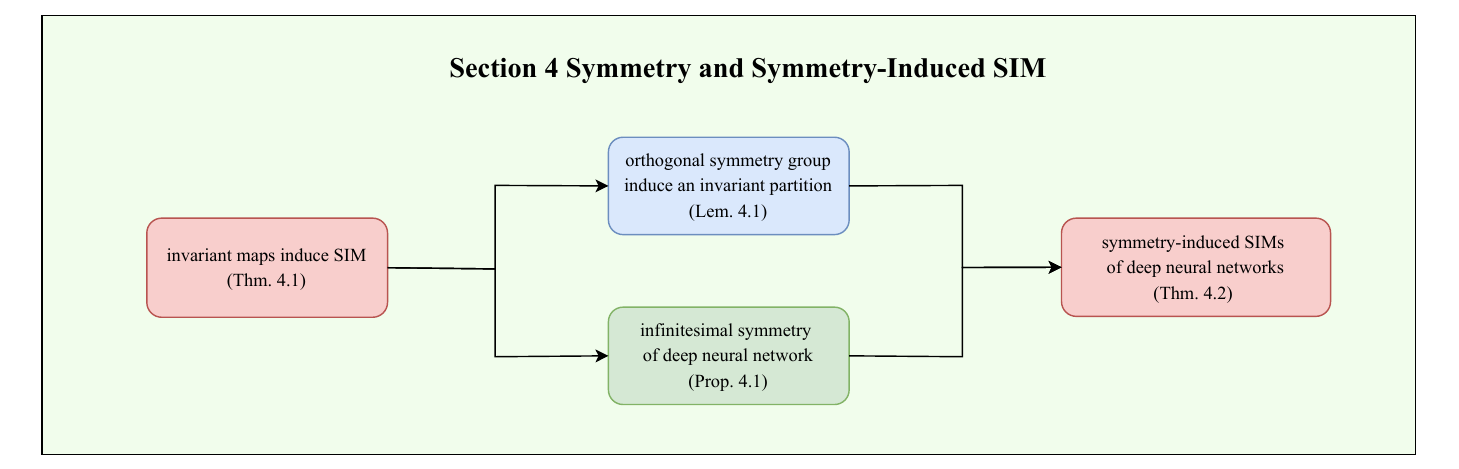}
    \caption{Flowchart illustrating the logical structure of Section~\ref{sec:symmetry}.}
    \label{fig:sec_4}
\end{figure}

In this section, we explore how model symmetries can give rise to SIMs. The logical structure of the section is illustrated in Figure~\ref{fig:sec_4}. We begin by establishing a general theoretical result: invariant maps are sufficient to generate SIMs (Theorem~\ref{cor:symmetries induced sim}). Building on this foundation, we apply the framework to deep neural networks, leveraging Proposition~\ref{prop:infinitesimal symmetry of neural network} and Lemma~\ref{lemma: invariant partition} to identify the relevant symmetries. These developments culminate in Theorem~\ref{thm:all symmetry of neural network}, which provides an explicit construction of symmetry-induced SIMs in deep neural networks.

\subsection{Symmetry-induced SIM}

We begin by examining how invariant maps give rise to SIMs. In Definition~\ref{def:invariant maps}, we formally define two types of invariant maps: infinitesimal invariant maps and global invariant maps, which we collectively refer to as \textbf{invariant maps}. It is worth noting that every global invariant map is also an infinitesimal invariant map, but not vice versa.

The notion of infinitesimal invariant maps is introduced for the following reasons. First, the symmetries that give rise to SIMs often do not require global invariance; instead, it is sufficient for the invariance to hold in a local neighborhood of the manifold, or even merely at the level of tangent. This motivates the generalization from global to infinitesimal invariance. Second, in the context of neural networks, one  encounters invariant maps that are globally defined but possess only tangent invariance (Proposition~\ref{prop:infinitesimal symmetry of neural network}). Such maps are still capable of inducing SIMs
despite exhibiting only this weaker form of invariance.

\begin{definition}[infinitesimal and global  invariant map]
Let $F(\vtheta)(\vx)$ be an analytic parametric model with $\vtheta\in \sR^M$ and $\vx\in \sR^d$. For an analytic map $g:\sR^M\to \sR^M$, we say $g$ is an \textbf{infinitesimal invariant map} if $\fM:=\{\vtheta'\mid g(\vtheta')=\vtheta'\}\neq \emptyset$, and
for any $\vtheta\in \fM$, any $\vx\in \sR^d$, $D_{\vtheta}\left( F(g(\vtheta))(\vx)\right)=D_{\vtheta}\left(F(\vtheta)(\vx)\right)$. Here $D_{\vtheta}$ denotes the Jacobian matrix.
Moreover, if $F(g(\vtheta))(\vx)=F(\vtheta)(\vx),\forall \vtheta\in \sR^M,\vx\in \sR^d$, we say $g$ is a \textbf{global invariant map}.  
\label{def:invariant maps}    
\end{definition}

 Theorem~\ref{cor:symmetries induced sim} provides a set of general conditions under which the fixed-point set of  invariant maps forms a SIM. In contrast, Example~\ref{example:not sim} shows that invariant maps---even globally invariant ones---do not necessarily induce SIMs without the other conditions. 
Theorem~\ref{cor:symmetries induced sim} subsumes prior results such as the \(O\)-mirror symmetry in Ref.~\refcite{ziyin2023symmetry} and the symmetric loss in Ref.~\refcite{simsek2021geometry}, up to a subtle distinction: those works consider symmetries of the empirical loss function \(L(\vtheta)\), whereas we focus on symmetries of the parametric model. Ignoring this difference, our result can be viewed as a generalization of these earlier cases. Moreover, Theorem~\ref{cor:symmetries induced sim} applies not only to linear but also to nonlinear invariant maps, as illustrated in Example~\ref{example:nonlinear symmetry}.

In addition to the symmetries considered in Theorem~\ref{cor:symmetries induced sim}, continuous symmetries, as discussed in Ref.~\refcite{ziyin2024parameter}, represent another class capable of inducing  SIMs. These continuous symmetries typically manifest in homogeneous networks and matrix factorization models. However, the scope of this paper is intentionally focused on the discrete symmetries detailed in Theorem $\text{\ref{cor:symmetries induced sim}}$ as they are generally shared by all neural networks.

\begin{theorem}[invariant maps induced SIM]
Let \( F(\vtheta)(\vx)\) be an analytic parametric model with $\vtheta\in \sR^M$ and $\vx\in \sR^d$. 
Let $\{g_i\}_{i\in I}$ be family of invariant maps of $F$. Define $\fM=\{\vtheta\mid g_i(\vtheta)=\vtheta,\forall i\in I\}$. Assume $\fM$ is an immersed submanifold of $\sR^M$ with its tangent space satisfying $T_{\vtheta}\fM=\bigcap_{i \in I} \ker(Dg_i^\T(\vtheta) - \mathrm{id}_M),\forall \vtheta\in \fM$.
 Then $\fM$ is a SIM.\footnote{$Dg_i$ is the jacobian matrix of $g_i$, and $Dg_i^\T$ is its transpose. $\mathrm{id}_M$ is the $M\times M$ identity matrix.}

\label{cor:symmetries induced sim}
\end{theorem}

\begin{proof}
By Theorem~\ref{def:sim},
it suffices to prove that for any \( \vx_0\in \sR^d,\vtheta_0 \in \fM \), the solution \( \vtheta(t) \) to the Cauchy problem \( \frac{\rd\vtheta}{\rd t} = \nabla_{\vtheta} F(\vtheta(t))(\vx_0), \vtheta(0) = \vtheta_0  \) remains in \( \fM \) for all \( t \) in its maximal interval of existence. Fix any $\vx_0\in \sR^d$, and define the vector field \( X(\vtheta) := \nabla_{\vtheta} F(\vtheta)(\vx_0) \). To apply Lemma~\ref{lemma:sim tangent}, we  now show that \( X(\vtheta) \in T_{\vtheta} \fM,\forall \vtheta\in \fM \).

Since a global invariant map is always an infinitesimal invariant map, without loss of generality we assume $g_i$ is an infinitesimal invariant map for each $i\in I$.
 Since $g_i$ is an infinitesimal invariant map, \( Dg_i(\vtheta)^\T \nabla_{\vtheta} F(g_i(\vtheta))(\vx_0) = \nabla_{\vtheta} F(\vtheta)(\vx_0) ,\forall \vtheta\in \fM\). When \( \vtheta \in \fM \), we have \( g_i(\vtheta) = \vtheta \), so the equation becomes \( Dg_i(\vtheta)^\T \nabla_{\vtheta} F(\vtheta)(\vx_0) = \nabla_{\vtheta} F(\vtheta)(\vx_0) \), implying \( \nabla_{\vtheta} F(\vtheta)(\vx_0) \in \ker(Dg_i(\vtheta)^\T - \mathrm{id}_M) \). Since this holds for all \( i \in I \), \( X(\vtheta) \in \bigcap_{i \in I} \ker(Dg_i(\vtheta)^\T - \mathrm{id}_M).\) By assumption,  \(\bigcap_{i \in I} \ker(Dg_i(\vtheta)^\T - \mathrm{id}_M)=T_{\vtheta}\fM\). Therefore $X(\vtheta)\in T_{\vtheta}\fM, \forall \vtheta \in \fM$.

Moreover, each fixed-point set \( \{ \vtheta \mid g_i(\vtheta) = \vtheta \} \) is closed (since $g_i$ is continuous), so \( \fM \), being their intersection, is closed in \( \mathbb{R}^M \).  Since \( X(\vtheta) \in T_{\vtheta} \fM \) and \( \fM \) is a closed immersed submanifold, it follows from Lemma~\ref{lemma:sim tangent}  that \( \vtheta(t) \in \fM \) for all \( t \) in the maximal interval of existence. So \(\fM \) is a SIM.
\end{proof}

\begin{remark}
\noindent

The assumption $\fM$ is an immersed submanifold and  \( T_{\vtheta}\fM = \bigcap_{i \in I} \ker(Dg_i^\T(\vtheta) - \mathrm{id}_M) ,\forall \vtheta\in \fM\) often holds under the following two conditions: 
\\(i)  \( \fM \) is an immersed submanifold  with \( T_{\vtheta}\fM = \bigcap_{i \in I} \ker(Dg_i(\vtheta) - \mathrm{id}_M) \), which is automatically satisfied when all \( g_i \) are linear maps.\\
(ii)  \( \bigcap_{i \in I} \ker(Dg_i^\T(\vtheta) - \mathrm{id}_M) = \bigcap_{i \in I} \ker(Dg_i(\vtheta) - \mathrm{id}_M) \) for all \( \vtheta \in \fM \), which is automatically satisfied when \( Dg_i(\vtheta) \) is a linear normal operator for all $i\in I,\vtheta\in \fM$.

As a direct corollary, the assumption $\fM$ is an immersed submanifold and $T_{\vtheta}\fM=\bigcap_{i \in I} \ker(Dg_i^\T(\vtheta) - \mathrm{id}_M),\forall \vtheta\in \fM$ holds if all \( g_i \) are linear normal  operators.
Besides,
 the regularity assumptions on $F(\vtheta)(\vx)$ and $g_i(\vtheta)$ can be weakened to $C^1$, and the domain of $\vtheta$ can be taken to be any open set $U \subset \sR^M$.

\label{rmk:orthogonal}  
\end{remark}

\begin{example}[global invariant map may not induce SIM]
Let \(\vtheta = (\theta_1, \theta_2) \in \sR^2\), and define \(F(\vtheta)(x) = (\theta_1 - \theta_2)x\) for all \(\vtheta \in \sR^2\), \(x \in \sR\). Consider the map \(g(\vtheta) = (\theta_1 + \theta_2, 2\theta_2)\). Then \(g\) is a global invariant map of \(F\). However, the fixed-point set of \(g\), given by \(\fM = \{ \vtheta \mid \theta_2 = 0 \}\), is not a SIM, since it is not invariant under all vector fields in \(\fF = \{ (x, -x) \mid x \in \sR \}\).

\label{example:not sim}
\end{example}

\begin{example}[nonlinear symmetry]  
Let \(F(\vtheta)(x) = ( \sqrt{\theta_1^2 + \theta_2^2} + \frac{1}{\sqrt{\theta_1^2 + \theta_2^2}} + x )^2\), where \(\vtheta = (\theta_1, \theta_2) \in \sR^2 \setminus \{0\}\) and \(x \in \sR\). Define the map \(g: \sR^2 \setminus \{0\} \to \sR^2 \setminus \{0\}\) by  
$
g(\theta_1, \theta_2) = ( \frac{\theta_1}{\theta_1^2 + \theta_2^2}, \frac{\theta_2}{\theta_1^2 + \theta_2^2}).
$  
Then \(g\) is a global invariant map of \(F\), and its fixed-point set is \(\fM = \{ (\theta_1, \theta_2) \in \sR^2 \mid \theta_1^2 + \theta_2^2 = 1 \}\). One can verify that the assumptions of Theorem~\ref{cor:symmetries induced sim} are satisfied. Therefore, \(\fM\) is a SIM.  
\label{example:nonlinear symmetry}  
\end{example}

Invariant maps can form a semigroup under composition. In neural networks, this semigroup is typically an \textbf{orthogonal symmetry group}, as defined in Definition~\ref{def:symmetry group_pre}.  
If a collection of SIMs forms a disjoint partition of the parameter space, we refer to this collection as an \textbf{invariant partition}. As shown in Lemma~\ref{lemma: invariant partition}, a finite orthogonal symmetry group can induce such an invariant partition, where each leaf corresponds to a set of parameters sharing the same stabilizer subgroup.

In the context of invariant partitions, a natural partial order can be defined based on partition coarseness: given two partitions \(\mathcal{P}_1\) and \(\mathcal{P}_2\) of a set, we say that \(\mathcal{P}_1\) is finer than \(\mathcal{P}_2\)  if every block of \(\mathcal{P}_1\) is contained within some block of \(\mathcal{P}_2\).  
Given an analytic model, the collection of orbits of \(\fF\) naturally forms an invariant partition. By Theorem~\ref{def:sim}, any SIM is a union of such orbits. It follows that the orbit partition is the finest invariant partition. Consequently, any invariant partition induced by an orthogonal symmetry group provides an upper bound for the orbit partition under this ordering.

\begin{lemma}[invariant partition induced by an orthogonal symmetry group]
Let \( F(\vtheta)(\vx) \) with \( \vtheta \in \mathbb{R}^M \) be an analytic model, and let \( G \) be an orthogonal symmetry group of finite elements. For each \( \vtheta \in \mathbb{R}^M \), define its stabilizer subgroup as
$$
S(\vtheta) := \{ g \in G \mid g(\vtheta) = \vtheta \}.
$$
Define an equivalence relation on the parameter space by
$
\vtheta_1 \sim \vtheta_2  \Longleftrightarrow  S(\vtheta_1) = S(\vtheta_2).
$
Denote by $[\vtheta]$ the equivalence class containing $\vtheta$. Then the collection $\{[\vtheta]\mid \vtheta\in \sR^M\}$ is an invariant foliation.
\label{lemma: invariant partition}
\end{lemma}

\begin{proof}
For any \( g \in G \), its fixed-point set \( \fM_g := \{ \vtheta \in \mathbb{R}^M \mid g(\vtheta) = \vtheta \} \) is a SIM by Theorem~\ref{cor:symmetries induced sim} and Remark~\ref{rmk:orthogonal}. A straightforward verification confirms that the defined relation satisfies the properties of an equivalence relation. Thus, it suffices to prove that for any $\vtheta\in \sR^M$, $[\vtheta]$ is a SIM. Fix any $\vtheta\in \sR^M$.
By definition,  \( [\vtheta] \) consists of all parameters whose stabilizer is exactly \( S(\vtheta) \). This allows \( [\vtheta] \) to be expressed in terms of set operations on the family \( \{ \fM_g \mid g \in G \} \) as follows:
\[
[\vtheta] = \left( \bigcap_{h \in S(\vtheta)} \fM_h \right) \cap \left( \bigcap_{g \in G \setminus S(\vtheta)} (\mathbb{R}^M \setminus \fM_g) \right).
\]
Since the family of structural invariant sets is closed under finite intersection and complement (Proposition~\ref{lemma:closed under set operations}),  \( [\vtheta] \) is a structural invariant set.

Moreover, since \( G \) has  finite elements, the set \( [\vtheta] \) can be viewed as the linear space \( \bigcap_{h \in S(\vtheta)} \fM_h \) with finitely many linear subspaces \( \left(   \bigcap_{h \in S(\vtheta)} \fM_h \right)\cap\fM_g\) (for \( g \notin S(\vtheta) \)) removed. Consequently, \( [\vtheta] \) is relatively open in the linear subspace \( \bigcap_{h \in S(\vtheta)} \fM_h \), and hence is an immersed submanifold of \( \sR^M \). Therefore, $[\vtheta]$ is a SIM. 

\end{proof}

\subsection{Symmetries of neural networks}

Symmetries are prevalent in deep neural networks, as detailed in Proposition~\ref{prop:infinitesimal symmetry of neural network} and Theorem~\ref{thm:all symmetry of neural network} below.
These symmetries generally originate from two principal sources.
First, the  indistinguishability of neurons within a given layer gives rise to the \textbf{permutation symmetry group}, denoted $G_{\mathrm{per}}$ (Theorem~\ref{thm:all symmetry of neural network}).
Second, the symmetry of the activation function, $\sigma(x)$, constitute another source.
Specifically, if $\sigma(x)$ possesses definite parity (i.e., is an odd or even function), a \textbf{reflection symmetry group}, $G_{\mathrm{sign}}$ or $G_{\mathrm{sign}}'$, emerges as an orthogonal symmetry group (Theorem~\ref{thm:all symmetry of neural network}).
Given that global invariant maps and orthogonal maps are closed under composition, the permutation and reflection group can generate more complex orthogonal symmetry groups under composition.
Furthermore, local symmetries can also be identified.
If $\sigma(0)=0$ or $\sigma'(0)=0$, the activation function $\sigma(x)$ exhibits infinitesimal odd or even behavior in the vicinity of the origin.
This local property results in the actions of elements in $G_{\mathrm{sign}}$ or $G_{\mathrm{sign}}'$ manifesting as infinitesimal symmetry maps (Proposition~\ref{prop:infinitesimal symmetry of neural network}).

To present Proposition~\ref{prop:infinitesimal symmetry of neural network} and Theorem~\ref{thm:all symmetry of neural network}, we first introduce Definition~\ref{def:symmetry group}. 
The group \( S_2^p \rtimes S_p \) in Definition~\ref{def:symmetry group} is also known as the hyper-octahedral group\cite{young1928quantitative}.

\begin{definition}
Consider the multi-layer neural network $F$ from Definition~\ref{def:fully connected network}, with layer widths $n_0,\ldots,n_L$. For any positive integer $p$, let $S_2^p$ denote the group of $p\times p$ diagonal sign matrices (entries in $\{\pm 1\}$), and let $S_p$ denote the group of $p\times p$ permutation matrices. Define the semidirect product \( S_2^p \rtimes S_p \) with the group operation  given by
\[
(\boldsymbol{\Lambda}_1, \boldsymbol{P}_1)(\boldsymbol{\Lambda}_2, \boldsymbol{P}_2) = \left( \boldsymbol{\Lambda}_1 \boldsymbol{P}_1 \boldsymbol{\Lambda}_2 \boldsymbol{P}_1^\top,\, \boldsymbol{P}_1 \boldsymbol{P}_2 \right),
\]
where \( \boldsymbol{\Lambda}_1, \boldsymbol{\Lambda}_2 \in S_2^p \) and \( \boldsymbol{P}_1, \boldsymbol{P}_2 \in S_p \).
One can readily verify that this structure satisfies the axioms of a semidirect product.
We define the following groups and describe their action on the parameter space of 
$F$:

\begin{enumerate}
\item

Define the group $G_{\mathrm{per}}=S_{n_1}\times \cdots \times S_{n_{L-1}}$, where $\times$ denotes the direct sum. For any $(\vP^{(1)},\ldots,\vP^{(L-1)})\in G_{\mathrm{per}}$, define its action  on the parameter space  as
    $$(\vP^{(1)},\ldots,\vP^{(L-1)}): \left(\vW^{(l)}, \vb^{(l)}\right)_{l=1}^L \mapsto \left(\vP^{(l)}\vW^{(l)}\vP^{(l-1)^\T},\vP^{(l)} \vb^{(l)}\right)_{l=1}^L,$$
    with the conventions $\vP^{(0)}=\mathrm{id}_{n_0}$ and $\vP^{(L)}=\mathrm{id_{n_L}}$.

    \item

    Define the group $G_{\mathrm{sign}}=S_2^{n_1} \times \cdots \times S_2^{n_{L-1}}$. For any  $(\vLambda^{(1)},\ldots,\vLambda^{(L-1)})\in G_{\mathrm{sign}}$, define its action  on the parameter space  as
\[
(\vLambda^{(1)},\ldots,\vLambda^{(L-1)}): \left(\vW^{(l)}, \vb^{(l)}\right)_{l=1}^L \mapsto \left(\vLambda^{(l)}\vW^{(l)}\vLambda^{(l-1)}, \vLambda^{(l)} \vb^{(l)}\right)_{l=1}^L,
\]
with the conventions $\vLambda^{(0)}=\mathrm{id}_{n_0}$ and $\vLambda^{(L)}=\mathrm{id_{n_L}}$.

\item

Define the group $G_{\mathrm{sign}}'=S_2^{n_1} \times \cdots \times S_2^{n_{L-1}}$. For any $(\vLambda^{(1)},\ldots,\vLambda^{(L-1)})\in G_{\mathrm{sign}}'$, define its action  on the parameter space  as
\[
(\vLambda^{(1)},\ldots,\vLambda^{(L-1)}): \left(\vW^{(l)}, \vb^{(l)}\right)_{l=1}^L \mapsto \left(\vLambda^{(l)}\vW^{(l)}, \vLambda^{(l)} \vb^{(l)}\right)_{l=1}^L,\]
with the conventions $\vLambda^{(L)}=\mathrm{id_{n_L}}$.

\item

Define the group $G_{\mathrm{combine}}=(S_2^{n_1}\rtimes S_{n_1})\times\cdots\times(S_2^{n_{L-1}}\rtimes S_{n_{L-1}})$. For $g=((\vLambda^{(1)},\vP^{(1)}),\ldots,(\vLambda^{(L-1)},\vP^{(L-1)}))\in G_{\mathrm{combine}}$, define its action on parameter space as
$$
g:\ \bigl(\vW^{(l)}, \vb^{(l)}\bigr)_{l=1}^L \mapsto \bigl(\vLambda^{(l)} \vP^{(l)} \vW^{(l)} \vP^{(l-1)^\T} \vLambda^{(l-1)},\, \vLambda^{(l)} \vP^{(l)} \vb^{(l)}\bigr)_{l=1}^L,
$$
with the conventions $\vP^{(0)}=\vLambda^{(0)}=\mathrm{id}_{n_0},\vP^{(L)}=\vLambda^{(L)}=\mathrm{id_{n_L}}$. 

\end{enumerate}

\label{def:symmetry group}
\end{definition}

\begin{proposition}[infinitesimal symmetry of deep neural networks]
Consider the multi-layer neural network from Definition~\ref{def:fully connected network}, with layer widths $n_0,\ldots,n_L$. Let $G_{\mathrm{sign}}$ and $G_{\mathrm{sign}}'$ be the groups  defined in Definition~\ref{def:symmetry group}. Then the following statements hold: 
\begin{enumerate}

\item If $\sigma(0)=0$,  then the action of any element in $G_{\mathrm{sign}}$ is an infinitesimal invariant map.
\item If $\sigma'(0)=0$, 
then the action of any element in $G_{\mathrm{sign}}'$ is an infinitesimal invariant map.
\end{enumerate}
\label{prop:infinitesimal symmetry of neural network}
\end{proposition}

\begin{proof}
    The proof is provided in Appendix~\ref{sec:proof of infinitesimal symmetry}.
\end{proof}

\begin{theorem}[symmetry-induced SIMs of deep neural networks]
Consider the multi-layer neural network from Definition~\ref{def:fully connected network}, with layer widths $n_0,\ldots,n_L$. Consider the groups and actions as   defined in Definition~\ref{def:symmetry group}. Then the following statements hold: 
\begin{enumerate}
    \item  $G_{\mathrm{per}}$ is an orthogonal symmetry group and thus induces an invariant partition.
\item If $\sigma(x)$ is an odd function, then $G_{\mathrm{sign}}$  is an orthogonal symmetry group. Moreover, $G_{\mathrm{sign}}$ and $G_{\mathrm{per}}$ generate a new orthogonal symmetry group under map composition, which  equals  $G_{\mathrm{combine}}$. Therefore, $G_{\mathrm{combine}}$ induces an invariant partition.
\footnote{The case in which $\sigma$ is even is analogous; we omit it for brevity.}

\item Assume $\sigma(0)=0$.  Let $I_0$ and $I_L$  be empty sets.  For any choice of subsets $I_l\subset \{1,\ldots,n_l\}$ for each $l\in\{1,\ldots,L-1\}$, define
    \[
    \mathcal{M} = \left\{ \left(\vW^{(l)}, \vb^{(l)}\right)_{l=1}^L \ \middle| \ 
    \begin{aligned}
    &\text{ } \vW^{(l)}_{ij} = 0, \quad \forall l \in \{1, \ldots, L\}, \ (i, j) \in I_l \times I_{l-1}^c \cup I_l^c \times I_{l-1}, \\
    &\text{ } \vb_i^{(l)} = 0, \quad \forall l \in \{1, \ldots, L\}, \ i \in I_l, \\
    &\text{ } \vW^{(l)}_{ij} = c_{ij}^{(l)}, \quad \forall l \in \{1, \ldots, L\}, \ (i, j) \in I_l \times I_{l-1}.
    \end{aligned}
    \right\},
    \]
where each \( c_{ij}^{(l)} \) is an arbitrary real number, $\vW_{ij}^{(l)}$ represents the matrix entry of $\vW^{(l)}$
 at position $(i,j)$, and $I_l^c=\{1,\ldots,n_l\}\setminus I_l$.
Then $\fM$ is a SIM.
\item If $\sigma'(0)=0$, then for each $l\in\{1,\ldots,(L-1)\}$ and $j\in\{1,\ldots,n_l\}$, the set
    $$\mathcal{M}_{l,j} = \left\{  \left(\vW^{(k)}, \vb^{(k)}\right)_{k=1}^L \mid \vW_{j}^{(l)} = \vzero \text{ and } \vb_j^{(l)} = 0 \right\}$$
    is a SIM. Here, $\vW_{j}^{(l)}$ denotes the $j$-th row of $\vW^{(l)}$.
\end{enumerate}
\label{thm:all symmetry of neural network}
\end{theorem}

\begin{proof}
(i) Let $\vP=(\vP^{(1)},\ldots,\vP^{(L-1)})$ and $\vP'=(\vP'^{(1)},\ldots,\vP'^{(L-1)})$ be elements of $G_{\mathrm{per}}$. Then the composition of their actions is given by
$$\vP'\circ \vP :\left(\vW^{(l)}, \vb^{(l)}\right)_{l=1}^L \mapsto \left(\vP'^{(l)}\vP^{(l)}\vW^{(l)}(\vP'^{(l-1)}\vP^{(l-1)})^\T,\vP'^{(l)}\vP^{(l)} \vb^{(l)}\right)_{l=1}^L,$$
which is exactly the action of $\vP'\vP$. 
Therefore, it is a group action. 

 Recall  $\va^{(l)}=\sigma(\vW^{(l)}\va^{(l-1)}+\vb^{(l)})$ for $l=1,\ldots,L-1$, and  $\va^{(L)}=\vW^{(L)}\va^{(L-1)}+\vb^{(L)}.$  We claim that the
 action of $\vP$ changes $\va^{(l)}$ to $\vP^{(l)} \va^{(l)}$ for $l=0,1,\ldots,L$. For $l=0$, $\va^{(0)}=\vx$. Since $\vP^{(0)}=\mathrm{id_{n_0}}$, the claim holds when $l=0$.
  Suppose this holds for some $l\in \{0,1,\ldots,L-1\}$.  If $l<L-1$, 
  then $\va^{(l+1)}$ is changed to 
$\sigma(\vP^{(l+1)}\vW^{(l+1)}\vP^{(l)^\T} \vP^{(l)} \va^{(l)}+\vP^{(l+1)}\vb^{(l+1)})=\sigma(\vP^{(l+1)}(\vW^{(l+1)}\va^{(l)}+\vb^{(l+1)}))=\vP^{(l+1)}\sigma(\vW^{(l+1)}\va^{(l)}+\vb^{(l+1)})=
\vP^{(l+1)}\va^{(l+1)}$.
In a similar manner, one can readily verify that \( \va^{(l+1)} \) is changed to \( \vP^{(l+1)}\va^{(l+1)} \) when \( l = L - 1 \). In both cases, $\va^{(l+1)}$ is changed to $\vP^{(l+1)}\va^{(l+1)}.$
By mathematical induction, the action of $\vP$ changes $\va^{(l)}$ to $\vP^{(l)} \va^{(l)}$ for $l=0,1,\ldots,L.$ 
Since $\vP^{(L)}=\mathrm{id}_{n_L}$, the action of $\vP$ does not change the output of the model.
Moreover, the permutation of coordinates is linear and does not change the norm of a vector. Therefore $G_{\mathrm{per}}$ is an orthogonal symmetry group.

(ii) Similar to the proof of (i), one can verify that the any action of element in  $G_{\mathrm{sign}}$ is a linear orthogonal operator that does not change the output of the model. So $G_{\mathrm{sign}}$ is an orthogonal symmetry group.
 The check that $G_{\mathrm{combine}}$ is the group generated by $G_{\mathrm{sign}}$ and $G_{\mathrm{per}}$ is straightforward, and we omit the details. Because both global invariant maps and orthogonal maps are closed under composition, \(G_{\mathrm{combine}}\) is also an orthogonal symmetry group.

(iii) In the proof of the first statement of Proposition~\ref{prop:infinitesimal symmetry of neural network}, one sees that the set
\[ \fM'=\left\{ \left(\vW^{(l)}, \vb^{(l)}\right)_{l=1}^L \ \middle| \ 
\begin{aligned}
&\text{ } \vW^{(l)}_{ij} = 0, \quad \forall l \in \{1, \ldots, L\}, \ (i, j) \in I_l \times I_{l-1}^c \cup I_l^c \times I_{l-1}, \\
&\text{ } \vb_i^{(l)} = 0, \quad \forall l \in \{1, \ldots, L\}, \ i \in I_l.
\end{aligned}
\right\}
\]
is the fixed point of $\vLambda=(\vLambda^{(1)},\ldots, \vLambda^{(L-1)})$, where $\vLambda^{(l)}$ is the diagonal matrix with entries equal to $-1$ for row indices in $I_l$ and $+1$ otherwise.
By Proposition~\ref{prop:infinitesimal symmetry of neural network}, $\vLambda$ is an infinitesimal invariant map. Since $\vLambda$ is linear and orthogonal,
by Theorem~\ref{cor:symmetries induced sim} and Remark~\ref{rmk:orthogonal}, 
$\fM'$ is a SIM.
By the proof of the first statement of Proposition~\ref{prop:infinitesimal symmetry of neural network}, the $i$-th row and $j$-th column of
$\frac{\partial F}{\partial \vW^{(l)}}$ are zero for any $i\in I_l, j\in I_{l-1}$. Therefore, for any $i\in I_l, j\in I_{l-1}$,
$\vW^{(l)}_{ij}$ remains constant during training. Therefore $\fM$ is a SIM.

(iv) Fix any $l\in\{1,\ldots,(L-1)\}$ and $j\in\{1,\ldots,n_l\}$.
Define $\vLambda^{(l)}$ to be the $n_l\times n_l$ diagonal matrix such that the diagonal satisfies $\vLambda^{(l)}_{jj}=-1$ and $\vLambda^{(l)}_{ii}=1, i\neq j$. Define $\vLambda=(\mathrm{id}_{n_1},\ldots,\mathrm{id}_{n_{l-1}},\vLambda^{(l)},\mathrm{id}_{n_{l+1}},\ldots, \mathrm{id}_{n_{(L-1)}})$.
By Proposition~\ref{prop:infinitesimal symmetry of neural network}, $\vLambda$ is an infinitesimal invariant map. It is easy to see that $\vLambda$ is an orthogonal linear map. By Theorem~\ref{cor:symmetries induced sim} and Remark~\ref{rmk:orthogonal}, the fixed point of $\vLambda$ is a SIM. Since the fixed point of $\vLambda$ is $\fM_{l,j}$, $\fM_{l,j}$ is a SIM.

\end{proof}

\begin{remark}
The symmetries in Theorem~\ref{thm:all symmetry of neural network} can be extended to other architectures, such as ResNet, Convolutional Neural Networks, and Transformers. 
\end{remark}

\section{Orbits of Two-layer Neural Networks}
\label{sec:orbit}

As established in Theorem~\ref{thm:all symmetry of neural network}, symmetry-induced SIMs exist within neural networks. This naturally prompts the question: \textbf{Do these symmetry-induced SIMs generate all possible SIMs in neural networks?}
Here, ``generate'' refers to the ability to obtain all SIMs of the neural network through set operations (unions, intersections, and complements).

In this section, we address this question in the context of generic two-layer neural networks and provide an affirmative answer. 
As shown in Figure~\ref{fig:sec_5},
our strategy  proceeds in three steps: 
\begin{enumerate}
    \item Develop a neuron independence result to calculate the rank of the Lie closure at non-degenerate $\vtheta$ (Lemma~\ref{lemma:neuron independence}, Corollary~\ref{cor:lie rank of non-degenerate}).
    \item Establish a perturbation lemma to transform the degenerate cases into non-degenerate  whenever possible (Lemma~\ref{lemma:pertubation_new}). This lemma allows us to calculate the rank of the Lie closure at all $\vtheta$ (Corollary~\ref{cor:pertube to non-degenerate},~\ref{cor:all rank of lie closure}). 
    \item Analyze the connectivity of the leaves of invariant partitions (Corollary~\ref{cor:connected}).
\end{enumerate}
These preparatory results, in conjunction with Theorem~\ref{theorem: page 44 book}, establish that all SIMs are symmetry-induced for generic two-layer neural networks (Theorem~\ref{thm:final orbit}).

\begin{figure}[h]
    \centering
    \includegraphics[width=0.95\linewidth]{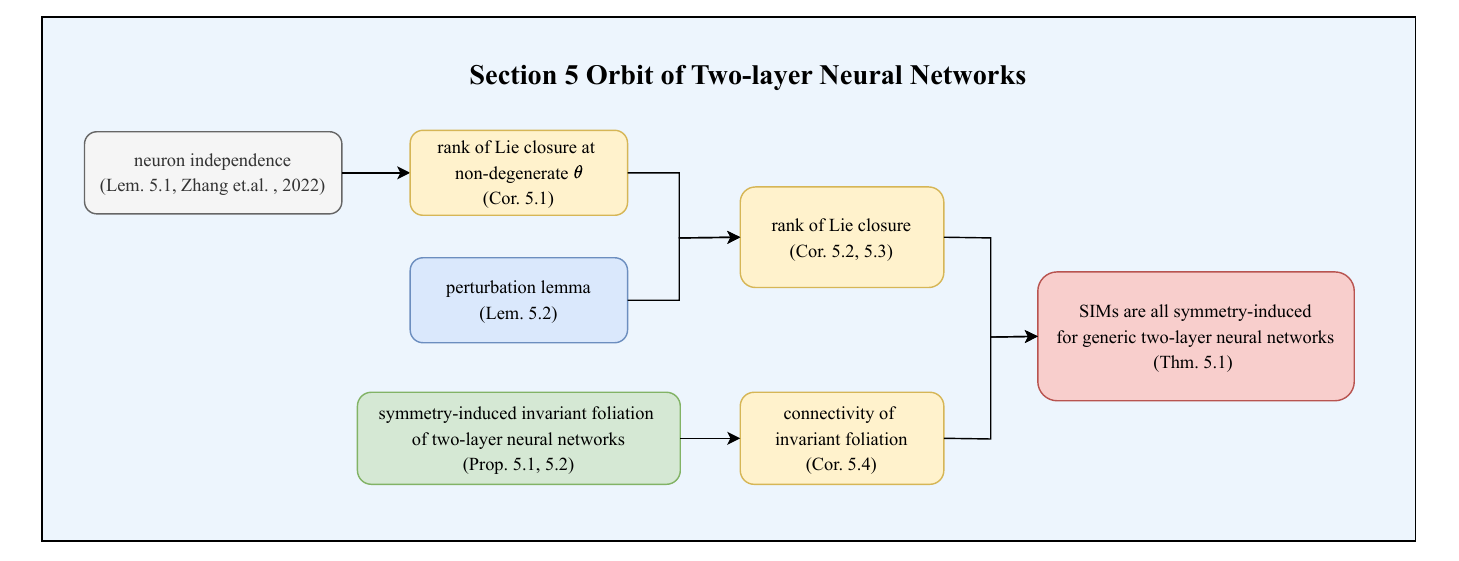}
    \caption{Flowchart illustrating the logical structure of Section~\ref{sec:orbit}.}
    \label{fig:sec_5}
\end{figure}

The two-layer neural network studied in this section is defined in Definition~\ref{def:neural network}. The following two preparatory  propositions characterize the invariant partitions induced by the permutation symmetry group $G_{\mathrm{per}}$ and the combined symmetry groups $G_{\mathrm{combine}}$  in Definition~\ref{def:symmetry group}.

\begin{proposition}[invariant partition induced by the permutation symmetry group]
Consider the two-layer neural network of width $m$ and the group $G_{\mathrm{per}}$  in Definition~\ref{def:symmetry group}. By Theorem~\ref{thm:all symmetry of neural network}, $G_{\mathrm{per}}$ induces an invariant partition. 
Let \( \mathfrak{P}_m \) denote the set of all partitions of \( \{1,\ldots,m\} \). For each partition \( \mathcal{P} = \{B_1, \ldots, B_s\} \in \mathfrak{P}_m \), define
\[
\mathcal{M}_{\mathcal{P}} := \left\{ (a_i, \vw_i)_{i=1}^m \in \mathbb{R}^{(d+1)m} \,\middle|\, 
\begin{aligned}
&(a_i, \vw_i) = (a_j, \vw_j), &&\forall p\in \{1,\ldots,s\}, \forall i,j \in B_p \\
&(a_i, \vw_i) \ne (a_j, \vw_j), &&
\forall p,p'\in \{1,\ldots,s\}, p\neq p', \forall i\in B_p,j\in B_{p'}
\end{aligned}
\right\}.
\]
Then the collection \( \{ \mathcal{M}_{\mathcal{P}} \mid \mathcal{P} \in \mathfrak{P}_m \} \) equals the invariant partition induced by $G_{\mathrm{per}}$.
\label{example:sim induced by permutation symmetry group}
\end{proposition}

\begin{proof}

 The proof is provided in Appendix~\ref{sec:proof of sim induced by permutation symmetry}.

\end{proof}

\begin{proposition}[invariant partition induced by the combined symmetry group]
Consider the two-layer neural network of width $m$
with odd activation function. Consider
 the group $G_{\mathrm{combine}}$ in Definition~\ref{def:symmetry group}. By
 Theorem~\ref{thm:all symmetry of neural network}, $G_{\mathrm{combine}}$ induces an invariant partition. Let \( \mathfrak{P}_m \) denote the set of all partitions of \( \{1,\ldots,m\} \). For each partition \( \mathcal{P} = \{B_1, \ldots, B_s\} \in \mathfrak{P}_m \) ($B_1 $ can be empty) and each \( \vgamma = (\gamma_1,\ldots,\gamma_m) \in \{-1,1\}^m \), define
\[
\mathcal{M}_{\mathcal{P}, \vgamma} := \left\{ (a_i, \vw_i)_{i=1}^m \in \mathbb{R}^{(d+1)m} \,\middle|\, 
\begin{aligned}
&(a_i,\vw_i) = \mathbf{0}, && \forall i \in B_1, \\
&\gamma_i(a_i, \vw_i) = \gamma_j(a_j, \vw_j), && \forall p \in \{2,\ldots,s\},\forall i,j \in B_p ,  \\
&(a_i, \vw_i) \ne \pm(a_j, \vw_j), && \forall p,p'\in \{1,\ldots,s\}, p\neq p', \forall i\in B_p,j\in B_{p'}
\end{aligned}
\right\}.
\]
Then the collection \( \{ \mathcal{M}_{\mathcal{P}, \vgamma} \mid \mathcal{P} \in \mathfrak{P}_m,\ \vgamma \in \{-1,1\}^m \} \) equals the invariant partition induced by $G_{\mathrm{combine}}$.
\label{example:sim induced by combined symmetry of tanh network}
\end{proposition}

\begin{proof}
 
We provide the proof in Appendix~\ref{sec:proof of SIM induced by the combined}. 
 
\end{proof}

\subsection{Rank of Lie closure}

We begin by presenting a lemma that establishes the foundation for the rank analysis carried out in this subsection.

\begin{lemma}[neuron independence \cite{zhang2022linear}] 
 Let $\sigma: \mathbb{R} \rightarrow \mathbb{R}$ be any analytic function such that $\sigma^{\left(n_j\right)}(0) \neq 0$ for an infinite sequence of distinct indices $\left\{n_j\right\}_{j=1}^{\infty}$. Given $d \in \mathbb{N}$ and $m$ distinct weights $\boldsymbol{w}_1, \ldots, \boldsymbol{w}_m \in \mathbb{R}^d \backslash\{\boldsymbol{0}\}$, such that $\boldsymbol{w}_k \neq \pm \boldsymbol{w}_j$ for all $1 \leqslant k<j \leqslant m$. Then $\left\{\sigma\left(\boldsymbol{w}_i^\T \boldsymbol{x}\right), \sigma^{\prime}\left(\boldsymbol{w}_i^\T \boldsymbol{x}\right) x_1, \ldots, \sigma^{\prime}\left(\boldsymbol{w}_i^\T \boldsymbol{x}\right) x_d\right\}_{i=1}^m$ is a linearly independent function set.

 \label{lemma:neuron independence}
 \end{lemma}
The linear independence established in Lemma~\ref{lemma:neuron independence} holds only when parameters satisfy certain conditions (e.g., $\vw_i=\pm \vw_j$
 ). We formally define parameters that do not meet these conditions or have some $a_i=0$ as \textbf{degenerate} in Definition~\ref{def:degenerate}. The definition of degeneracy is designed to ensure  the Lie closure of $\fF$ at non-degenerate $\vtheta$ has full rank, as stated in Corollary~\ref{cor:lie rank of non-degenerate}.

\begin{definition}[degenerate and non-degenerate]
Consider the two-layer neural network. 
For $\vtheta=(a_i,\vw_i)_{i=1}^m$, if $\vtheta$ satisfies (i): $a_k\neq 0 , \vw_k\neq \vzero$ for all $ k\in\{1,\ldots,m\}$, and (ii): $\vw_i\neq \pm \vw_j$ for any $i,j\in\{1,\ldots,m\}$ and $i\neq j$, then  $\vtheta$ is said to be \textbf{non-degenerate}. Otherwise $\vtheta$ is said to be  \textbf{degenerate}.   \label{def:degenerate}
\end{definition}

\begin{corollary}
Consider the two-layer neural network, and suppose that \( \vtheta \in \sR^M \) is non-degenerate. Then \( \dim(\mathrm{Lie}_{\vtheta}(\fF)) = (d+1)m \).

\label{cor:lie rank of non-degenerate}
\end{corollary}
\begin{proof}
Let $\vtheta = (a_i, \vw_i)_{i=1}^m$ be non-degenerate. By calculation,
$
\nabla_{\vtheta} F(\vtheta)(\vx) = \left( \sigma(\vw_i^\T\vx), a_i\sigma'(\vw_i^\T\vx)x_1, \ldots, a_i\sigma'(\vw_i^\T\vx)x_d\right)_{i=1}^m.
$
To show that $\dim(\mathrm{Lie}_{\vtheta}(\fF)) = (d+1)m$, we will prove that the set of vectors $U := \{ \nabla_{\vtheta} F(\vtheta)(\vx) \mid \vx \in \mathbb{R}^d \}$ spans the entire parameter space $\mathbb{R}^{(d+1)m}$.

We proceed by contradiction. Assume that $\mathrm{span}(U)$ is a proper subspace of $\mathbb{R}^{(d+1)m}$. Then there must exist a non-zero constant vector $\vc = (c_1, \mathbf{v}_1^\T, \ldots, c_m, \mathbf{v}_m^\T)^\T \in \mathbb{R}^{(d+1)m}$, where $c_i \in \mathbb{R}$ and $\mathbf{v}_i=(v_{i,1}, \ldots, v_{i,d}) \in \mathbb{R}^d$, that is orthogonal to every vector in $U$. This orthogonality condition, $\vc^\T \nabla_{\vtheta} F(\vtheta)(\vx) = 0$ for all $\vx \in \mathbb{R}^d$, expands to:
$$
\sum_{i=1}^m \left( c_i \sigma(\vw_i^\T \vx) + \sum_{j=1}^d (a_i v_{i,j}) \sigma'(\vw_i^\T \vx) x_j \right) = 0.
$$
This equation is a linear combination of the functions in the set $\left\{\sigma\left(\boldsymbol{w}_i^\T \boldsymbol{x}\right), \sigma^{\prime}\left(\boldsymbol{w}_i^\T \boldsymbol{x}\right) x_1, \ldots, \sigma^{\prime}\left(\boldsymbol{w}_i^\T \boldsymbol{x}\right) x_d\right\}_{i=1}^m$ that is identically zero for all $\vx \in \mathbb{R}^d$.

Since $\vtheta$ is non-degenerate, Lemma~\ref{lemma:neuron independence} guarantees that this set of functions is linearly independent. Consequently, all coefficients of the linear combination must be zero. This implies:
\begin{enumerate}
    \item $c_i = 0$ for all $i=1, \ldots, m$.
    \item $a_i v_{i,j} = 0$ for all $i=1, \ldots, m$ and $j=1, \ldots, d$.
\end{enumerate}
The non-degeneracy of $\vtheta$ ensures that $a_i \neq 0$ for all $i=1,\ldots,m$. From the second point, we must have $v_{i,j} = 0$ for all $i=1,\ldots,m, j=1,\ldots,d$, which means $\mathbf{v}_i = \mathbf{0}$ for all $i=1,\ldots,m$. This implies that the entire vector $\vc$ is the zero vector, which contradicts our assumption that $\vc$ was non-zero.

Therefore, the initial assumption must be false, and $\mathrm{span}(U) = \mathbb{R}^{(d+1)m}$. Since $\mathrm{span}(U)\subset \mathrm{Lie}_{\vtheta}(\fF)\subset \sR^{(d+1)m}$, $\dim(\mathrm{Lie}_{\vtheta}(\fF)) = (d+1)m$.
\end{proof}

\begin{remark}
Theorem~\ref{thm:all symmetry of neural network} implies that the degenerate case of $\vtheta$ can give rise to SIMs. For completeness, we provide its proof here.
\begin{enumerate}
\item For any $i,j\in\{1,\ldots,m\}$,  the set \(\{(a_k,\vw_k)_{k=1}^m \in \sR^{(d+1)m}\mid (a_i, \vw_i) = (a_j, \vw_j)\}\) is a SIM.
    \item If \(\sigma(x)\) is an odd function, then for any $i,j\in\{1,\ldots,m\}$,  the set \(\{(a_k,\vw_k)_{k=1}^m \in \sR^{(d+1)m}\mid (a_i, \vw_i) = -(a_j, \vw_j)\}\) is a SIM.
    \item  If \(\sigma(x)\) is an even function, then for any $i,j\in\{1,\ldots,m\}$,  the set \(\{(a_k,\vw_k)_{k=1}^m \in \sR^{(d+1)m}\mid (a_i, \vw_i) = (a_j, -\vw_j)\}\) is a SIM.
    \item If \(\sigma(0) = 0\), then for any $i\in\{1,\ldots,m\}$, the set \(\{(a_k,\vw_k)_{k=1}^m \in \sR^{(d+1)m}\mid(a_i, \vw_i) = \vzero\}\) is a  SIM.
    \item If \(\sigma'(0) = 0\), then for any $i\in\{1,\ldots,m\}$,  the set $\{(a_k,\vw_k)_{k=1}^m \in \sR^{(d+1)m}\mid \vw_i = \vzero\}$ is a SIM.
    
\end{enumerate}
\label{rmk:sim related to sigma}
\end{remark}

\begin{proof}
We prove the five statements item by item. The procedure is the same for each: we denote the set in question as a manifold $\fM$ and then verify that it is a SIM. Since each of these five sets is a linear subspace of  $\sR^{(d+1)m}$, by Lemma~\ref{lemma:sim tangent}, we only need to show that for any parameter $\vtheta \in \fM$ and any input $\vx \in \sR^d$, the gradient $\nabla_{\vtheta}F(\vtheta)(\vx)$ also lies in the tangent space of $\fM$, which for a linear subspace means $\nabla_{\vtheta}F(\vtheta)(\vx) \in \fM$.
For any $\vtheta=(a_k,\vw_k)_{k=1}^m \in \fM$ and $\vx \in \sR^d$, the $i$-th component of the gradient (corresponding to the parameters $(a_i, \vw_i)$ of neuron $i$) is given by:
\[
\nabla_i F(\vtheta)(\vx) = (\sigma(\vw_i^\T\vx), a_i\sigma'(\vw_i^\T\vx)\vx^\T).
\]
We now analyze each item:
\begin{enumerate}
    \item Define $\fM = \{(a_k,\vw_k)_{k=1}^m \mid (a_i, \vw_i) = (a_j, \vw_j)\}$.
    If $\vtheta \in \fM$, we have $(a_i, \vw_i) = (a_j, \vw_j)$. This directly implies that their corresponding gradient components are equal:
    \[
    \nabla_i F(\vtheta)(\vx) = (\sigma(\vw_i^\T\vx), a_i\sigma'(\vw_i^\T\vx)\vx^\T) = (\sigma(\vw_j^\T\vx), a_j\sigma'(\vw_j^\T\vx)\vx^\T) = \nabla_j F(\vtheta)(\vx).
    \]
    Thus, $\nabla_{\vtheta}F(\vtheta)(\vx) \in \fM$.

    \item Assume $\sigma(x)$ is an odd function, and define $\fM = \{(a_k,\vw_k)_{k=1}^m \mid (a_i, \vw_i) = -(a_j, \vw_j)\}$.
    We use the property that the derivative of an odd function is an even function, i.e., $\sigma'(-z) = \sigma'(z)$. If $\vtheta \in \fM$, we have $a_i = -a_j$ and $\vw_i = -\vw_j$. The $i$-th component of the gradient is:
    \begin{align*}
    \nabla_i F(\vtheta)(\vx) &= (\sigma(\vw_i^\T\vx), a_i\sigma'(\vw_i^\T\vx)\vx^\T) \\
    &= (\sigma(-\vw_j^\T\vx), -a_j\sigma'(-\vw_j^\T\vx)\vx^\T) \\
    &= (-\sigma(\vw_j^\T\vx), -a_j\sigma'(\vw_j^\T\vx)\vx^\T) && \text{(since $\sigma$ is odd, $\sigma'$ is even)} \\
    &= -(\sigma(\vw_j^\T\vx), a_j\sigma'(\vw_j^\T\vx)\vx^\T) = -\nabla_j F(\vtheta)(\vx).
    \end{align*}
    Thus, $\nabla_{\vtheta}F(\vtheta)(\vx) \in \fM$.

    \item Assume $\sigma(x)$ is an even function, and define $\fM = \{(a_k,\vw_k)_{k=1}^m \mid (a_i, \vw_i) = (a_j, -\vw_j)\}$.
     If $\vtheta \in \fM$, we have $a_i = a_j$ and $\vw_i = -\vw_j$. Let $\nabla_k F = (\nabla_{a_k}F, \nabla_{\vw_k}F)$. Then we have:
    \begin{align*}
    \nabla_i F(\vtheta)(\vx) &= (\sigma(\vw_i^\T\vx), a_i\sigma'(\vw_i^\T\vx)\vx^\T) \\
    &= (\sigma(-\vw_j^\T\vx), a_j\sigma'(-\vw_j^\T\vx)\vx^\T) \\
    &= (\sigma(\vw_j^\T\vx), a_j(-\sigma'(\vw_j^\T\vx))\vx^\T) && \text{(since $\sigma$ is even, $\sigma'$ is odd)} \\
    &= (\nabla_{a_j}F, -\nabla_{\vw_j}F).
    \end{align*}
    Thus, $\nabla_{\vtheta}F(\vtheta)(\vx) \in \fM$.

    \item Assume $\sigma(0) = 0$, and define $\fM = \{(a_k,\vw_k)_{k=1}^m \mid (a_i, \vw_i) = \vzero\}$.
    If $\vtheta \in \fM$, we have $(a_i, \vw_i) = (0, \vzero)$. The $i$-th component of the gradient is:
    \[
    \nabla_i F(\vtheta)(\vx) = (\sigma(\vzero^\T\vx), 0 \cdot \sigma'(\vzero^\T\vx)\vx^\T) = (\sigma(0), \vzero).
    \]
    Given the condition $\sigma(0) = 0$, this becomes $(0, \vzero)$. Thus, $\nabla_{\vtheta}F(\vtheta)(\vx) \in \fM$.

    \item Assume $\sigma'(0) = 0$, and define $\fM = \{(a_k,\vw_k)_{k=1}^m \mid \vw_i = \vzero\}$.
    If $\vtheta \in \fM$, we have $\vw_i = \vzero$. We only need to examine the component of the gradient corresponding to $\vw_i$:
    \[
    \nabla_{\vw_i}F(\vtheta)(\vx) = a_i\sigma'(\vw_i^\T\vx)\vx^\T = a_i\sigma'(\vzero^\T\vx)\vx^\T = a_i\sigma'(0)\vx^\T.
    \]
    Given the condition $\sigma'(0) = 0$, the expression becomes $a_i \cdot 0 \cdot \vx^\T = \vzero$. This satisfies the condition for the gradient vector to be in $\fM$.
\end{enumerate}
\end{proof}

To deal with the degenerate case, we introduce Lemma~\ref{lemma:pertubation_new}, which allows us to perturb those degenerate $\vtheta$ to non-degenerate  whenever possible.

\begin{lemma}[perturbation lemma]
Consider the two-layer neural network.
For $\vtheta^*=(a_i^*,\vw_i^*)_{i=1}^m$, the following statement holds:

\begin{enumerate}
    \item Assume $i\in\{1,\ldots,m\}$ and  $a_i^*=0$. Then $$\exists \delta>0, B_{\delta}(\vtheta^*)\cap   O_\fF(\vtheta^*) \subset \{(a_i,\vw_i)_{i=1}^m \mid a_i=0\} \Longleftrightarrow \vw_i^*=\vzero, \sigma(0)=0.$$
    \item Assume $i\in\{1,\ldots,m\}$ and  $\vw_i^*=\vzero, a_i^*\neq 0$. Then $$\exists \delta>0, B_{\delta}(\vtheta^*)\cap O_\fF(\vtheta^*) \subset \{(a_i,\vw_i)_{i=1}^m \mid \vw_i=\vzero\} \Longleftrightarrow \sigma'(0)=0.$$
    \item Assume $i,j\in\{1,\ldots,m\}$ and  $\vw_i^*=\vw_j^*\neq \vzero $. Then
$$\exists \delta>0, B_{\delta}(\vtheta^*)\cap O_\fF(\vtheta^*) \subset \{(a_i,\vw_i)_{i=1}^m \mid \vw_i=\vw_j\} \Longleftrightarrow a_i^*=a_j^*.$$
\item   Assume $i,j\in\{1,\ldots,m\}$ and  $\vw_i^*=-\vw_j^*\neq \vzero $. Then
\begin{align*}
&\exists \delta>0, B_{\delta}(\vtheta^*)\cap O_\fF(\vtheta^*) \subset \{(a_i,\vw_i)_{i=1}^m \mid \vw_i=-\vw_j\} \\
&\Longleftrightarrow \text{$a_i^*=-a_j^*,\sigma(x)$ is odd \text{ } or \text{ } $ a_i^*=a_j^*,\sigma(x)$ is even.}
\end{align*}
\end{enumerate}
Here, $B_{\delta}(\vtheta^*)$ denotes the open $\delta$-ball around $\vtheta^*$.
\label{lemma:pertubation_new}    
\end{lemma}
\begin{proof}
In each case, the goal is to characterize when the intersection \(B_{\delta}(\vtheta^*) \cap O_{\fF}(\vtheta^*)\) is contained within a specific linear subspace \(\fM\). 
Since
 \( O_{\fF}(\vtheta^*) \) is  immersed submanifold of \( \sR^M \), it follows that for all \( \vtheta \in B_{\delta}(\vtheta^*) \cap O_{\fF}(\vtheta^*) \), the tangent space \( T_{\vtheta}(B_{\delta}(\vtheta^*) \cap O_{\fF}(\vtheta^*)) = T_{\vtheta} O_{\fF}(\vtheta^*) \) is contained in \( T_{\vtheta} \fM=\fM \). By Theorem~\ref{thm:Hermann-Nagano}, \( T_{\vtheta} O_{\fF}(\vtheta^*) = \mathrm{Lie}_{\vtheta}(\fF) \), hence \( \mathrm{Lie}_{\vtheta}(\fF) \subset \fM \) for all \( \vtheta \in B_{\delta}(\vtheta^*) \cap O_{\fF}(\vtheta^*) \). 
 Since $\nabla_{\vtheta}F(\vtheta)(\vx)\in \mathrm{Lie}_{\vtheta} (\fF),\forall \vtheta\in \sR^M,\vx\in \sR^d$, we have
$\nabla_{\vtheta}F(\vtheta)(\vx)\in \fM,\forall \vtheta\in B_\delta(\vtheta^*)\cap O_{\fF}(\vtheta^*),\vx\in\sR^d$.
Particularly $\nabla_{\vtheta}F(\vtheta^*)(\vx)\in \fM, \forall \vx\in \sR^d$.
We analyze this condition for the four statements to be proved. For simplicity we use $\vtheta$ to denote $(a_i,\vw_i)_{i=1}^m\in \sR^{(d+1)m}$.

\textbf{(i)}
$\Longrightarrow$: Assume $O_\fF(\vtheta^*) \subset \{ \vtheta \mid a_i=0 \}$.
The condition $\nabla_{\vtheta}F(\vtheta^*)(\vx)\in  \{ \vtheta \mid a_i=0 \}, \forall \vx\in \sR^d$
indicates that $\sigma(\vw_i^{*\T} \vx) = 0$ for all $\vx \in \mathbb{R}^d$. Since $\sigma$ is a real, non-polynomial, analytic function, this can only hold if $\vw^*=\vzero$ and $\sigma(0)=0$.

$\Longleftarrow$: Assume $\vw_i^*=\vzero$ and $\sigma(0)=0$. By Remark~\ref{rmk:sim related to sigma}, when $\sigma(0)=0$, $\{ \vtheta \mid (a_i,\vw_i)=\vzero \}$ is a SIM. Therefore $O_\fF(\vtheta^*) \subset  \{ \vtheta \mid (a_i,\vw_i)=\vzero \}\subset \{ \vtheta \mid a_i=0 \}$.

\textbf{(ii)} $\Longrightarrow$: Assume $O_\fF(\vtheta^*) \subset \{ \vtheta \mid \vw_i=\vzero \}$. This requires $\dot{\vw}_i = a_i \sigma'(\vw_i^\T \vx) \vx$ to be zero at $\vtheta^*$ for all $\vx\in \sR^d$. So $a_i^* \sigma'(\vw_i^{*\T} \vx) \vx = \vzero$ for all $\vx \in \mathbb{R}^d$. Given $\vw_i^*=\vzero$, the equation becomes $a_i^* \sigma'(0) \vx = \vzero$. Since $a_i^* \neq 0$ and this must hold for all $\vx$, it follows that $\sigma'(0)=0$.

$\Longleftarrow$: Assume $\vw_i^*=\vzero$ and $\sigma'(0)=0$.  By Remark~\ref{rmk:sim related to sigma}, when $\sigma'(0)=0$, $\{ \vtheta \mid \vw_i=\vzero \}$ is a SIM. Therefore $O_\fF(\vtheta^*) \subset \{ \vtheta \mid \vw_i=\vzero \}$.

\textbf{(iii)}
$\Longrightarrow$ Assume $O_\fF(\vtheta^*) \subset \{ \vtheta \mid \vw_i=\vw_j \}$. Then $\dot{\vw}_i = \dot{\vw}_j$ at $\vtheta^*$ for all  $\vx\in \sR^d$. This means $a_i^* \sigma'(\vw_i^{*\T} \vx) \vx = a_j^* \sigma'(\vw_j^{*\T} \vx) \vx,\forall \vx\in \sR^d$. As $\vw_i^* = \vw_j^*$, this simplifies to $(a_i^*-a_j^*) \sigma'(\vw_i^{*\T} \vx) \vx = \vzero,\forall \vx\in \sR^d$ . Since $\vw_i^* \neq \vzero$ and $\sigma$ is not a zero function, the function $\sigma'(\vw_i^{*\T} \vx)\vx$ is not identically zero. Thus, we must have $a_i^*=a_j^*$.

$\Longleftarrow$: Assume $a_i^*=a_j^*$ and $\vw_i^*=\vw_j^*$. By the permutation symmetry introduced in Remark~\ref{rmk:sim related to sigma}, the set $\{\vtheta\mid a_i=a_j,\vw_i=\vw_j\}$ is a SIM. Therefore $O_\fF(\vtheta^*) \subset \{ \vtheta \mid a_i=a_j,\vw_i=\vw_j \}\subset \{\vtheta \mid \vw_i=\vw_j \}$.

\textbf{(iv)}
$\Longrightarrow$: Assume $O_\fF(\vtheta^*) \subset \{ \vtheta \mid \vw_i=-\vw_j \}$. Given $\vw_i^* = -\vw_j^*$, it requires $\dot{\vw}_i = -\dot{\vw}_j$ at $\vtheta^*$. This implies $a_i^* \sigma'(\vw_i^{*\T} \vx) \vx = -a_j^* \sigma'(\vw_j^{*\T} \vx) \vx = -a_j^* \sigma'(-\vw_i^{*\T} \vx) \vx$. Since this must hold for all $\vx\in \sR^d$ and $\vw_i^*\neq \vzero$, it follows that $a_i^* \sigma'(t) + a_j^* \sigma'(-t) = 0$ for all $t \in \mathbb{R}$. Integrating with respect to $t$ yields $a_i^* \sigma(t) - a_j^* \sigma(-t) = C$ for some constant $C$. Replacing $t$ with $-t$ gives $a_i^* \sigma(-t) - a_j^* \sigma(t) = C$. Equating the two expressions gives $(a_i^*+a_j^*)(\sigma(t)-\sigma(-t)) = 0$.

This implies two cases. First, if $\sigma(t) = \sigma(-t)$ for all $t$ (i.e., $\sigma$ is an even function), then $\sigma'$ is odd. The condition $a_i^* \sigma'(t) + a_j^* \sigma'(-t) = 0$ becomes $(a_i^*-a_j^*)\sigma'(t)=0$. As $\sigma'$ is not identically zero, we must have $a_i^*=a_j^*$.

Second, if $\sigma$ is not an even function, then we must have $a_i^* + a_j^* = 0$. Choose $0<\delta'<\delta$ such that for all $\vtheta'=(a_i',\vw_i')_{i=1}^m\in B_{\delta'}(\vtheta^*)$, we have $\vw_i'\neq \vzero$ and $\vw_j' \neq \vzero$. For any $\vtheta'\in B_{\delta'}(\vtheta^*)\cap O_{\fF}(\vtheta^*)$, there exists $\delta''$ such that $B_{\delta''}(\vtheta')\subset B_{\delta}(\vtheta^*)$. Thus, $B_{\delta''}(\vtheta')\cap O_{\fF}(\vtheta')\subset B_{\delta}(\vtheta^*)\cap O_{\fF}(\vtheta') =B_{\delta}(\vtheta^*)\cap O_{\fF}(\vtheta^*)
\subset \{\vtheta\mid \vw_i=-\vw_j\}$. Since $\sigma$ is not an even function, the same derivation implies $a_i'+a_j'=0$. Therefore, $B_{\delta'}(\vtheta^*)\cap O_{\fF}(\vtheta^*)\subset \{\vtheta\mid a_i=-a_j\}$. This implies $\dot{a_i}+\dot{a_j}=0$ at $\vtheta^*$ for any input $\vx\in \sR^d$. Thus, $\sigma(\vw_i^{*^\T}\vx)+\sigma(-\vw_i^{*^\T}\vx)=0$ for all $\vx\in \sR^d$. Since $\vw_i^{*}\neq \vzero$, $\sigma$ must be an odd function.

$\Longleftarrow$: Assume $\sigma$ is odd and $a_i^*=-a_j^*$, $\vw_i^*=-\vw_j^*$. By Remark~\ref{rmk:sim related to sigma} the submanifold $\{ \vtheta \mid a_i=-a_j, \vw_i=-\vw_j \}$ is a SIM. Therefore $O_\fF(\vtheta^*) \subset \{ \vtheta \mid a_i=-a_j, \vw_i=-\vw_j \} \subset\{ \vtheta \mid \vw_i=-\vw_j \}$. The case when $\sigma$ is even and $a_i^*=a_j^*$, $\vw_i^*=-\vw_j^*$ is similar.

\end{proof}

\begin{remark}
Lemma~\ref{lemma:pertubation_new} establishes that the \(\vtheta^*\) can be perturbed arbitrarily slightly outside the specified subset. For example, in the first case, if \(\vw_i^* \neq \vzero\) or \(\sigma(0) \neq 0\), then for any \(\epsilon > 0\), there exists a perturbed parameter  \(\vtheta' \in O_{\fF}(\vtheta^*)\) such that \(\|\vtheta' - \vtheta^*\|_2 < \epsilon\) and \(\vtheta' \notin \{(a_i, \vw_i)_{i=1}^m \mid a_i = 0\}\). This is why we refer to it as the perturbation lemma.
\end{remark}

As observed in Remark~\ref{rmk:sim related to sigma} and Lemma~\ref{lemma:pertubation_new}, there are numerous cases to consider for the activation function \(\sigma(x)\), such as whether \(\sigma(x)\) is odd or even, whether \(\sigma(0) = 0\), and whether \(\sigma'(0) = 0\). To manage this complexity, we focus on two representative types of activation: generic activation  and generic odd activation.

\begin{definition}[generic activation and generic odd activation]
A real analytic function \( \sigma: \mathbb{R} \to \mathbb{R} \) is called a \textbf{generic activation} if it satisfies the following conditions:  
(i): \( \sigma(x) \) is not a polynomial;  
(ii): \( \sigma(x) \) is neither an odd function nor an even function;  
(iii): \( \sigma(0) \neq 0 \) and \( \sigma'(0) \neq 0 \).

Similarly, a real analytic function \( \sigma: \mathbb{R} \to \mathbb{R} \) is called a \textbf{generic odd activation} if it satisfies the following conditions:  
(i): \( \sigma(x) \) is not a polynomial;  
(ii): \( \sigma(x) \) is an odd function;  
(iii): \( \sigma'(0) \neq 0 \).
\label{def:good activation}
\end{definition}

When the activation function is either a generic activation or a generic odd activation, any degenerate parameter \( \vtheta \) with trivial stabilizer subgroup $S(\vtheta)$ can always be perturbed to a non-degenerate one. Once such a perturbation is made, Theorem~\ref{thm:Hermann-Nagano} implies that the Lie closure at \( \vtheta \) has rank equal to \( (d+1)m \). This idea is illustrated in Corollary~\ref{cor:pertube to non-degenerate}.

\begin{corollary}
Consider the two-layer neural network. For  $\vtheta=(a_i,\vw_i)_{i=1}^m \in \sR^{(d+1)m}$, the following holds:
\begin{enumerate}
    \item Assume $\sigma(x)$ is a generic activation, and $(a_i,\vw_i)\neq (a_j,\vw_j)$ for any $i,j\in \{1,\ldots,m\}$ with $i\neq j$. Then there exists $\vtheta'\in O_\fF(\vtheta)$ such that $\vtheta'$ is non-degenerate. Thus, $\mathrm{dim(Lie_{\vtheta}(\fF)})=(d+1)m$.
    \item Assume $\sigma(x)$ is a generic odd activation, and $(a_i,\vw_i)\neq \pm(a_j,\vw_j)$ for any $i,j\in \{1,\ldots,m\}$. Then there exists $\vtheta'\in O_\fF(\vtheta)$ such that $\vtheta'$ is non-degenerate. Thus, $\mathrm{dim(Lie_{\vtheta}(\fF)})=(d+1)m$.
\end{enumerate}
\label{cor:pertube to non-degenerate}
\end{corollary}

\begin{proof}
(i) Assume that $\sigma(x)$ is generic activation, and assume $(a_i,\vw_i)\neq (a_j,\vw_j)$ for any $i,j\in \{1,\ldots,m\}, i\neq j$. We now prove that there exists $\vtheta'\in O_\fF(\vtheta)$ such that $\vtheta'$ is non-degenerate.
Since $\sigma(0)\neq 0$, by the first statement of  Lemma~\ref{lemma:pertubation_new}, there exists $\vtheta_1=(a_i^1,\vw_i^1)_{i=1}^m\in O_\fF(\vtheta)$ such that $a_i^1\neq 0$ for all $i\in \{1,\ldots,m\}$. Moreover,  $\vtheta_1$ can be arbitrary close to $\vtheta$ such that  $(a_i^1,\vw_i^1)\neq (a_j^1,\vw_j^1)$ for any $i,j\in \{1,\ldots,m\}, i\neq j$. Since $\vtheta_1$ and $\vtheta$ are on the same orbit, we
regard them as equivalent. In sense of this equivalence, 
without loss of generality we can assume that $a_i\neq 0$ for all $i\in \{1,\ldots,m\}$. By the second statement of Lemma~\ref{lemma:pertubation_new}, without loss of generality we can assume that $\vw_i\neq \vzero$ for all $i\in \{1,\ldots,m\}$. If there exists $i\neq j$ such that $\vw_i=\vw_j$, then $a_i\neq a_j$. By
the third statement of Lemma~\ref{lemma:pertubation_new}, 
without loss of generality we can assume that $\vw_i\neq \vw_j$ for all $i,j\in \{1,\ldots,m\}$ and $i\neq j$.  By the fourth statement of Lemma~\ref{lemma:pertubation_new} we can assume that $\vw_i\neq -\vw_j$ for all $i\neq j$. Therefore there exists non-degenerate $\vtheta'$
on the orbit of $\vtheta$.

(ii) Assume that  $\sigma(x)$ is  generic odd activation, and for any $i,j\in \{1,\ldots,m\}$, $(a_i,\vw_i)\neq \pm(a_j,\vw_j)$. Therefore for each $i\in \{1,\ldots,m\}$, either $a_i\neq 0$ or $\vw_i\neq \vzero$. If $a_i\neq 0$, by the second statement of Lemma~\ref{lemma:pertubation_new}, without loss of generality we can assume that $\vw_i\neq \vzero$. If $\vw_i\neq \vzero$, by the first statement of Lemma~\ref{lemma:pertubation_new}, without loss of generality we can assume that $a_i\neq 0$. In both cases we can assume that $a_i\neq 0, \vw_i\neq \vzero$ for all $i\in \{1,\ldots,m\}$. If there exists $i\neq j$ such that $\vw_i=\vw_j$, then $a_i\neq a_j$. By
the third statement of Lemma~\ref{lemma:pertubation_new}, 
without loss of generality we can assume that $\vw_i\neq \vw_j$ for all $i,j\in \{1,\ldots,m\}$ and $i\neq j$. Similarly, if there exists $i\neq j$ such that $\vw_i=-\vw_j$, then $a_i\neq -a_j$.
By the fourth statement of Lemma~\ref{lemma:pertubation_new}, without loss of generality we can assume that $\vw_i\neq -\vw_j$ for all $i,j\in \{1,\ldots,m\}$ and $i\neq j$. Therefore there exists a non-degenerate $\vtheta'$ on the orbit of $\vtheta$. 

In both cases there exists a non-degenerate $\vtheta'\in O_\fF(\vtheta)$. By Corollary~\ref{cor:lie rank of non-degenerate}, $\mathrm{dim(Lie_{\vtheta'}(\fF)})=(d+1)m$. By Theorem~\ref{thm:Hermann-Nagano}, $\mathrm{dim(Lie_{\vtheta}(\fF)})=\mathrm{dim(Lie_{\vtheta'}(\fF)})=(d+1)m$.

\end{proof}

The result of Corollary~\ref{cor:pertube to non-degenerate} extends to any parameter \( \vtheta \) within a leaf of the invariant partition induced by the orthogonal symmetry group, as formalized in Corollary~\ref{cor:all rank of lie closure}.  Throughout and after Corollary~\ref{cor:pertube to non-degenerate}, the equivalence class $[\vtheta]$ is defined according to Lemma~\ref{lemma: invariant partition} as follows.

 For the orthogonal symmetry group $G$, the stabilizer subgroup of a parameter \( \vtheta \in \mathbb{R}^{(d+1)m} \) is given by $S(\vtheta) := \{ g \in G \mid g(\vtheta) = \vtheta \}$. An equivalence relation on the parameter space is then defined by $ \vtheta_1 \sim \vtheta_2 \Longleftrightarrow S(\vtheta_1) = S(\vtheta_2) $, and $[\vtheta]$ denotes the resulting equivalence class containing $\vtheta$. For generic and generic odd activation functions, $G$ corresponds to the permutation symmetry group $G_{\mathrm{per}}$ and the combined symmetry group $G_{\mathrm{combine}}$ defined in Definition~\ref{def:symmetry group}, respectively.

\begin{corollary}[rank of Lie closure]
Consider the two-layer neural network. Assume that $\sigma(x)$ is generic activation or generic odd activation. Then for any $\vtheta\in \sR^{(d+1)m}$, $\mathrm{dim(Lie_{\vtheta}(\fF)})$ is equal to the dimension of $[\vtheta]$.
\label{cor:all rank of lie closure}
\end{corollary}

\begin{proof}

The proof is provided in Appendix~\ref{sec:proof of all rank of lie closure}.

\end{proof}

\subsection{Orbits}

Corollary~\ref{cor:all rank of lie closure} establishes that, on each leaf of the invariant partition, the rank of the Lie closure equals the dimension of the leaf itself. In this case, Theorem~\ref{theorem: page 44 book}  indicates that determining the orbit reduces to verifying the connectivity of each leaf. In Corollary~\ref{cor:connected}, we show that every leaf is indeed connected.

\begin{corollary}
Consider the two-layer neural network, and assume that \( \sigma(x) \) is either a generic activation or a generic odd activation. Then, for all \( \vtheta \in \mathbb{R}^{(d+1)m} \), the equivalence class \( [\vtheta] \) is connected.
\label{cor:connected}
\end{corollary}

\begin{proof}
We begin by presenting a standard lemma from manifold geometry\cite{Lecture}.
\begin{lemma}[derived from Theorem 1.1 in Ref.~\refcite{Lecture}]
Let \( A_1, \ldots, A_n \) be linear subspaces of \( \mathbb{R}^M \), each with codimension at least 2. Then the set \( \mathbb{R}^M \setminus \bigcup_{i=1}^n A_i \) is connected.
\label{lemma:connected 1}
\end{lemma}

We now return to the proof of the corollary.
As shown in Propositions~\ref{example:sim induced by permutation symmetry group} and~\ref{example:sim induced by combined symmetry of tanh network}, for any \( \vtheta \in \mathbb{R}^M \), the corresponding leaf \( [\vtheta] \) takes the form $
[\vtheta] = B \setminus \left( \bigcup_{i=1}^s A_i \right),
$  
where \( B \subset \mathbb{R}^M \) is a linear subspace, and each \( A_i \subset B \) is a linear subspace of \( B \) with \( \dim(A_i) \leq \dim(B) - 2 \).  
By Lemma~\ref{lemma:connected 1}, such a set \( [\vtheta] \) is connected.
\end{proof}

With the necessary preliminaries established, we are now in a position to present Theorem~\ref{thm:final orbit}, which serves as the principal result of this section.
\begin{theorem}[SIMs are all symmetry-induced
for generic two-layer neural networks]
Consider the two-layer neural network, and assume that \( \sigma(x) \) is either a generic activation or a generic odd activation. Then for any \( \vtheta \in \mathbb{R}^{(d+1)m} \), $[\vtheta]=O_{\fF}(\vtheta)$.
\label{thm:final orbit}
\end{theorem}

\begin{proof}
For any $\vtheta \in \mathbb{R}^{(d+1)m}$, Corollary~\ref{cor:all rank of lie closure} implies that $\dim(\mathrm{Lie}_{\vtheta}(\fF)) = \dim([\vtheta])$, where $\dim([\vtheta])$ denotes the dimension of $[\vtheta]$. By Corollary~\ref{cor:connected}, the set $[\vtheta]$ is connected. According to Lemma~\ref{lemma: invariant partition},
$[\vtheta]$ is a SIM. So $\fF$ can be regarded as a family of vector fields defined on $[\vtheta]$. 
Applying Theorem~\ref{theorem: page 44 book}, it then follows that $[\vtheta]=O_{\fF}(\vtheta),\forall \vtheta\in \sR^{(d+1)m}$.
\end{proof}

\begin{remark}
Theorem~\ref{thm:final orbit} establishes that the symmetry-induced SIMs identified in Theorem~\ref{thm:all symmetry of neural network} are sufficient to generate all SIMs of generic two-layer neural networks via set-theoretic operations such as union, intersection, and complement.
Moreover, we conjecture that this result extends to arbitrary two-layer neural networks, as defined in Definition~\ref{def:neural network}. A proof may be attainable by extending the techniques developed in this section.

\end{remark}

\begin{figure}[htbp]
    \centering    \includegraphics[width=0.8\linewidth]{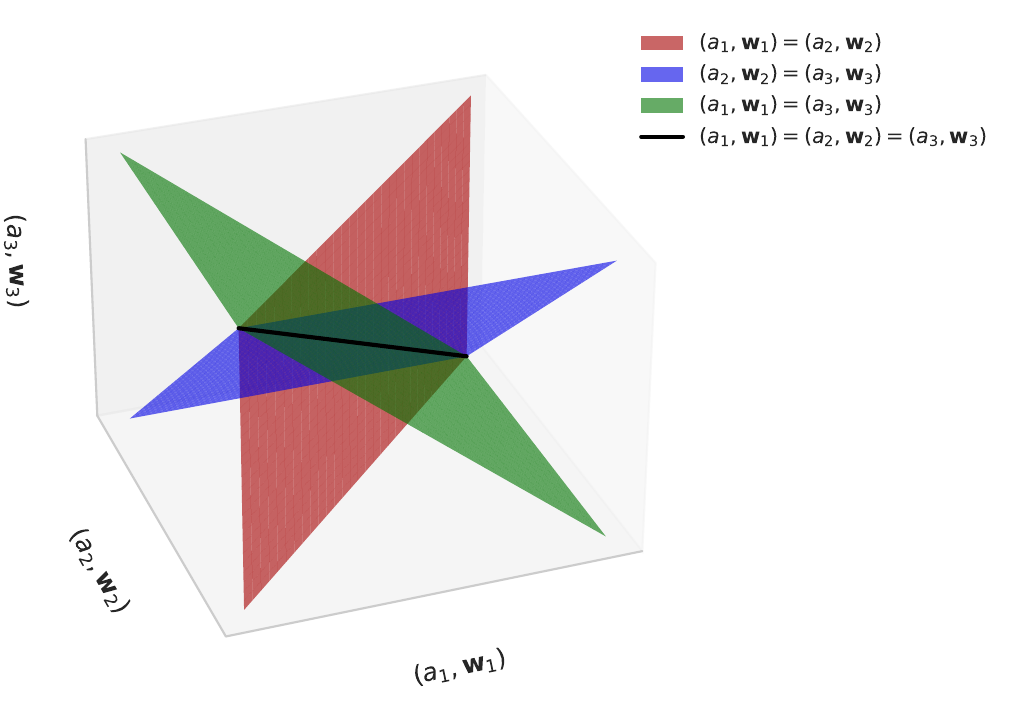}
    \caption{ We visualize four representative symmetry-induced SIMS in a width-three, two-layer neural network with a generic activation. To enable 3D visualization, for each \( i = 1, 2, 3 \), we compress the parameter pair \((a_i, \mathbf{w}_i)\) into a single coordinate. }
    \label{fig:illustration}
\end{figure}

\begin{example}
Consider a two-layer neural network with sigmoid activation and width $3$. The network function is given by
$
F(\vtheta)(\vx) = \sum_{i=1}^3 a_i \sigma(\vw_i^\T \vx),
$
where the activation function is defined as \(\sigma(x) = \frac{1}{1 + e^{-x}}\). 
In Figure~\ref{fig:illustration}, we visualize  four representative symmetry-induced SIMs for this network. These are defined as follows:
\[
\begin{aligned}
&\fM_1 = \left\{ \vtheta \mid (a_1, \vw_1) = (a_2, \vw_2) \right\}, \quad 
\fM_2 = \left\{ \vtheta \mid (a_2, \vw_2) = (a_3, \vw_3) \right\}, \\
&\fM_3 = \left\{ \vtheta \mid (a_1, \vw_1) = (a_3, \vw_3) \right\}, \quad 
\fM_4 = \left\{ \vtheta \mid (a_1, \vw_1) = (a_2, \vw_2) = (a_3, \vw_3) \right\}.
\end{aligned}
\]

By Proposition~\ref{example:sim induced by permutation symmetry group}, each equivalence class \([\vtheta]\) can be generated from these four symmetry-induced SIMs using set operations (intersection, union, and complement). 
Furthermore, by Theorem~\ref{thm:final orbit}, for any \(\vtheta \in \mathbb{R}^M\), \([\vtheta]\) is  equal to the orbit  passing through $\vtheta$. 
Additionally, Theorem~\ref{def:sim} states that each SIM is a union of orbits. As a consequence, all SIMs of the network can be constructed from these four representative symmetry-induced SIMs using intersection, union, and complement operations. 
This provides a constructive interpretation of the statement “all SIMs are symmetry-induced” as asserted in Theorem~\ref{thm:final orbit}.
\end{example}

\section{Discussion}

\subsection{Relevance to Practical Training Method}
\label{sec:SGD}

While the results of this paper primarily concern the theoretical analysis of gradient-flow dynamics, they hold implications for more practical training methods such as gradient descent (GD) and stochastic gradient descent (SGD). 

\paragraph{Affine SIM.}

The symmetry-induced structural invariant manifolds (SIMs) in deep neural networks, as characterized in Theorem~\ref{thm:all symmetry of neural network}, are all affine subspaces. Theoretically, it is straightforward to show that such affine SIMs remain invariant under both gradient descent (GD) and stochastic gradient descent (SGD) updates. To empirically support this claim, we conduct an experiment using a two-layer neural network. The experimental results are presented in Figure~\ref{fig:invariance}, with additional details provided in Appendix~\ref{sec:details of invariance}.

In our experiment, we consider a two-layer neural network with width $100$. The network is initialized on a SIM by setting the parameters of the first two hidden neurons to be identical, i.e., \(a_1(0) = a_2(0)\) and \(\vw_1(0) = \vw_2(0)\). To examine the data-independence of SIM invariance, we adopt a training scheme where, at each iteration, a input-target pair is sampled from a standard normal distribution. To test loss-independence, we perform three independent training runs using different loss functions.

Throughout training, we track the Euclidean distance \(\|\boldsymbol{\theta}_1(t) - \boldsymbol{\theta}_2(t)\|_2\) between the parameters of the two initially identical neurons. As shown in Figure~\ref{fig:invariance}, this distance remains below \(10^{-50}\) across all 2000 training steps, effectively zero at every iteration. This invariance holds across all loss functions tested, and persists even when training data is resampled at each step. The experimental results confirm that SIMs are invariant under both GD and SGD, with this invariance holding regardless of the data and the  loss function. 

\begin{figure}[htbp]
    \centering
    \includegraphics[width=0.55\linewidth]{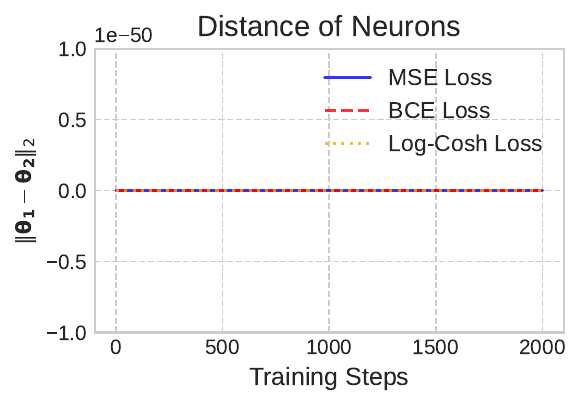} 
\caption{ The plot shows the Euclidean distance \(\|\boldsymbol{\theta}_1 - \boldsymbol{\theta}_2\|_2\) between two identically initialized neurons. The details of the experiment are presented in Appendix~\ref{sec:details of invariance}.}
\label{fig:invariance}
\end{figure}

\paragraph{Curved SIM.}
In  the matrix factorization model and ReLU networks, certain conserved quantities give rise to curved SIMs~\cite{simsek2021geometry,ziyin2023symmetry,marcotte2023abide,marcotte2025intrinsic}. Unlike affine SIMs, these curved SIMs are not necessarily invariant under GD or SGD. Specifically, the use of a finite learning rate in GD induces dynamics in directions normal to the SIM, while the stochasticity inherent in SGD introduces additional perturbations in the normal direction~\cite{ziyin2024parameter}. However, for both GD and SGD, these normal dynamics tend to vanish as the learning rate approaches zero, thereby recovering the invariance property under the gradient flow limit~\cite{ziyin2024parameter}.

\subsection{Dynamics Near SIM.} 

In this work, we establish that, once the parameters are initialized on SIM, the gradient flow dynamics remain confined to it and cannot escape. However, in neural networks, symmetry-induced SIMs constitute a Lebesgue measure zero subset of the parameter space. Consequently, under random initialization, it is highly improbable for the parameters to lie exactly on the SIM. To bridge this gap, we analyze the behavior of the dynamics in the neighborhood of the SIM.

Given that the SIM is invariant under gradient flow, the continuity of the dynamics implies that, for initializations sufficiently close to the manifold, the gradient component in the direction normal to the SIM must be small. As a result, the trajectory drifts away from the SIM only gradually, remaining in its vicinity for an extended period. As a consequence, the SIM functions as a ``slow manifold," where the dynamics in the normal direction evolve significantly more slowly than those along the manifold.

\begin{figure}[htbp]
    \centering
    \includegraphics[width=1\linewidth]{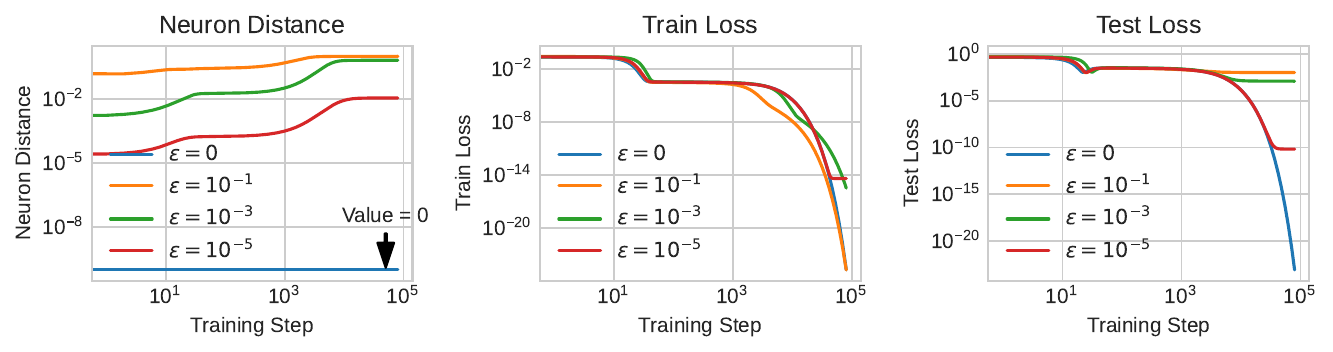}
    \caption{ We train a two-layer neural network of width $2$ under different initialization. The input space is two-dimensional. $\varepsilon$ represents the deviation from the SIM $\{\vtheta\mid (a_1,\vw_1)=(a_2,\vw_2)\}.$
     The target function is a width-$1$ two-layer neural network with all weights set to $1$. Training data consists of four input vectors \(\mathbf{x}\) sampled randomly from a standard Gaussian distribution, and the corresponding outputs \(y\) are generated by passing \(\mathbf{x}\) through the target network. The neuron distance is the value of $\|\vtheta_1-\vtheta_2\|_2$.
    To evaluate the test loss, we randomly sampled 100 input vectors \( \mathbf{x} \) from a standard normal distribution and measured the \( \ell_2 \) distance between the outputs of the network and the target function.
    }
\label{fig:normal_dynamics}
\end{figure}

To empirically support this claim, we conduct a controlled experiment using a two-layer neural network with input dimension $2$ and hidden layer width $2$. To investigate the behavior of the network near SIM, we initialize the two hidden neurons to be nearly identical, differing only by a small perturbation of magnitude \(\varepsilon\). Specifically, we consider \(\varepsilon = 0\), corresponding to exact initialization on the SIM, as well as \(\varepsilon = 10^{-1}, 10^{-3}, 10^{-5}\), to systematically explore the effect of deviations from the manifold. The initialization scale of parameters is $0.1$.

Details of the experimental setup are provided in Appendix~\ref{sec:details of normal dynamics}. The corresponding results are presented in Figure~\ref{fig:normal_dynamics}. We highlight the following two key observations in Figure~\ref{fig:normal_dynamics},:

\noindent\textbf{Persistence in the vicinity of the SIM.}  
When the network is initialized slightly off the SIM (e.g., \(\varepsilon = 10^{-3}, 10^{-5}\)), we observe that the distance between the two hidden neurons gradually increases over time. However, this increase remains small throughout a significant portion of training, indicating that the trajectory stays close to the SIM for an extended period. This empirical behavior aligns well with our theoretical prediction regarding the slow dynamics in directions normal to the SIM.

\noindent\textbf{Low Test loss in the overparameterized setting.}  
We also examine the test loss throughout training. Note that this is an overparameterized regime: the model contains six trainable parameters but is trained on only four data samples. These samples are synthetically generated by a teacher network with a single hidden neuron.  As shown in Figure~\ref{fig:normal_dynamics}, the final test loss is  remarkably low when \(\varepsilon = 0\) and \(\varepsilon = 10^{-5}\), indicating good generalization despite the overparameterization.

When initialized exactly on the SIM (\(\varepsilon = 0\)), the two hidden neurons remain identical throughout training, effectively reducing the number of independent parameters to three. In this case, the model has just enough capacity to fit the training data, enabling it to accurately learn the target function.

When initialized sufficiently close to the SIM (e.g., \(\varepsilon = 10^{-5}\)), the network remains in the vicinity of the SIM for a substantial portion of training. During this phase, the two neurons behave almost identically, and the model effectively operates with approximately three degrees of freedom. This implicit regularization---arising from the slow escape from the SIM---contributes to good generalization performance on the one-neuron target function, even in the presence of overparameterization.

\section{Conclusions}
\label{sec:conclusion}

In this work, we lay the theoretical foundation for identifying SIMs in analytic parametric models by employing geometric control theory. By uncovering the hierarchy of symmetry-induced SIMs for deep neural networks and enumerating all SIMs for the two-layer network, we unravel the profound dynamical consequence of the layer-wise neural network architecture. These SIMs display condensation behavior and underlie the remarkable potential for target recovery in overparameterized settings. Although the milestone of fully solving the recovery puzzle for neural networks has not yet been reached, these SIMs offer powerful tools for tracing global training dynamics. Building on our findings, we anticipate major breakthroughs in solving the recovery puzzle in the near future. Such advances will pave the way for a comprehensive generalization theory that clarifies how architecture design, target properties, training samples, nonlinear dynamics, and parameter tuning collectively shape the generalization of neural networks.

\newpage

\renewcommand{\thesection}{A} 
\setcounter{section}{0}      
\renewcommand{\thedefinition}{A.\arabic{definition}}
\setcounter{definition}{0}
\renewcommand{\theremark}{A.\arabic{remark}} 
\setcounter{remark}{0}

\appendix

\section{Definitions}
\label{appendix:def}

In the appendix, we present several definitions and concepts from geometric control theory that are pertinent to the content of this paper. As our analysis is conducted within the analytic category, all definitions are stated in their analytic form. The material is drawn from Ref.~\refcite{jurdjevic1997geometric} and Ref.~\refcite{ortega2013momentum}.
\subsection{Differential Geometry}

\begin{definition}[analytic manifold, page 3 and 4 of Ref.~\refcite{jurdjevic1997geometric})]
$\fM$ is called an $n$ dimensional analytic manifold if $\fM$ is a topology space such that at each point $p\in \fM$
 there exists a neighbourhood $U$ of $p$ and a homeomorphism $\phi$ from $U$ onto an open subset of $\sR^n$.  It is assumed that $n$ does not vary with the choice of a point $p$ on $\fM$. The pair $(\phi,U)$ is called a chart at $p$. Moreover:
 \begin{enumerate}
 \item There exists a countable collection of charts $\{(\phi_i,U_i)\}_{i=1}^{\infty}$ such that $\fM=\bigcup_{i=1}^{\infty} U_i$.
     \item For each pair of points $p_1$ and $p_2$, there exist charts $(\phi_1,U_1)$ and $(\phi_2,U_2)$ such that $p_1\in U_1$, $p_2\in U_2$, and $U_1\cap U_2=\emptyset$. That is, points of $\fM$ are separated by coordinate neighborhoods (i.e., $\fM$ is Hausdorff).
        \item  For any charts $(\phi_1,U_1)$ and $(\phi_2,U_2)$ such that $U_1\cap U_2\neq \emptyset$, the mapping $\phi_1\circ \phi_2^{-1}$ is analytic as a mapping from an open set in $\sR^n$ into $\sR^n$.

 \end{enumerate}
\end{definition}

\begin{definition}[analytic vector fields, Definition 1 in Chapter 1 of Ref.~\refcite{jurdjevic1997geometric}] Let $\fM$ be an analytic manifold. The totality of $(p,v),p\in \fM, v\in T_p \fM$, is called the tangent bundle of $\fM$ and is denoted by $T\fM$. A vector field is a mapping $X: \fM \to T\fM$ such that for each $p \in \fM$, if $\pi: T\fM \to \fM$ denotes the natural projection, then $\pi\big(X(p)\big)=p$. We say that $X$ is an analytic vector field if $X$ is an analytic map from $\fM$ (as an analytic manifold) into $T\fM$ (another analytic manifold).
\label{def: analytic vector fields}
\end{definition}

\begin{definition}[integral curve,  Definition 3 in Chapter 1 of Ref.~\refcite{jurdjevic1997geometric}]
A differential curve $p(t),t\in J$ on $\fM$ is an integral curve of an analytic vector field $X$ if $\frac{\rd p}{\rd t}=X\circ p$ for each $t$ in $J$. We shall say an integral curve $p(t),t\in J$ of $X$ is the integral curve through $p_0\in \fM$ if $p(0)=p_0$ and the domain $J\subset \sR$ is maximal.
    
\end{definition}

\begin{definition}[complete vector field, flow, $e^{tX}$,    Definition 4 in Chapter 1 of Ref.~\refcite{jurdjevic1997geometric}]
A vector field $X$ is called complete if the integral curves through each point $p_0$ in $\fM$ are defined for all values of $t$ in $\mathbb{R}$. In such a case, $X$ is said to define a flow $\Phi$ on $\fM$. $\Phi: \mathbb{R} \times \fM \rightarrow \fM$ is defined by $\Phi(t, p_0) = p(t)$, where $p(t)$ is the integral curve through $p_0$.
For each $t$, define the mapping $\Phi_t(p)=\Phi(t,p)$. We shall also use $e^{tX}$ to denote the mapping $\Phi_t$.
\end{definition}
\begin{remark}[page 16 of Ref.~\refcite{jurdjevic1997geometric}]
For complete vector fields, its flow $\Phi$ has following properties:
\begin{enumerate}
    \item $\Phi(0, p) = p$ for all $p \in \fM$.
    \item $\Phi(t+s, p) = \Phi(t, \Phi(s, p))$ for all $(s, t)$ in $\mathbb{R}^2$ and all $p \in \fM$.
    \item $(\partial/\partial t)\Phi(t, p) = X \circ \Phi(t, p)$ for all $(t, p)$ in $\mathbb{R} \times \fM$.
    \item The mapping $\Phi$ is analytic  whenever $X$ is analytic.
    \item For each $t$, $e^{tX}$ is a diffeomorphism on $\fM$.
\end{enumerate}

\end{remark}

\begin{definition}[local flow, page 17 of Ref.~\refcite{jurdjevic1997geometric}]
Let $X$ be an analytic vector (possibly non-complete) field on an analytic manifold $\fM$. 
In order to define the local flow of a vector field at $p$ in $\fM$, it is first necessary to define the escape times of the integral curve of $X$ through $p$. The positive escape time $e^+(p)$ is defined to be the supremum of the domain of the
integral curve through $p$. The negative escape time $e^-(p)$ is defined similarly. Let $\Delta = \{(t, p) \mid e^-(p) < t < e^+(p)\}$. Then $\Delta$ is an open subset of $\mathbb{R} \times \fM$ and a neighborhood of $\{0\} \times \fM$. The local flow $\Phi$ of $X$ is defined on $\Delta$.
\end{definition}
\begin{remark}[page 17 of Ref.~\refcite{jurdjevic1997geometric}]
The local flow $\Phi$ satisfies the following:

\begin{enumerate}
    \item $\Phi(0, p) = p$ for all $p \in \Delta$.
    \item $\Phi(t+s, p) = \Phi(t, \Phi(s, p))$ whenever each of $(s, p)$ and $(t, \Phi(s, p))$ is contained in $\Delta$.
    \item $(\partial\Phi/\partial t)(t, p) = X \circ \Phi(t, p)$ for all $(p, t)$ in $\Delta$.
    \item $\Phi$ is analytic whenever $X$ is analytic.
\end{enumerate}
    
\end{remark}
\begin{definition}[immersed submanifold, Definition 1 in Chapter 2 of Ref.~\refcite{jurdjevic1997geometric}]
 An differentiable mapping
 $f$ between two differential manifolds is called an immersion if the rank of
 the tangent map of $f$ at each point is equal to the dimension of the domain manifold. Then the definition of an immersed submanifold  is as follows:
Given two differentiable manifolds $\fM$ and  $\fN$, if there exists
an immersion $f:\fN\to \fM$, then $f(\fN)$ is called an immersed submanifold of $\fM$.
 
\label{def:immersed submanifold}
\end{definition}

\subsection{Local Diffeomorphisms and Pseudogroups}
\label{appendix:local diffeomorphism}

The materials of~\ref{appendix:local diffeomorphism} are  from Section 3.1 of Ref.~\refcite{ortega2013momentum}.
\subsubsection{Local diffeomorphisms }
Let $\fM$ be an  analytic manifold. The symbol $\mathrm{Diff}(\fM)$ will denote the set of diffeomorphisms of $\fM$.
The symbol $\mathrm{Diff}_L(\fM)$ will denote the set of local diffeomorphisms of $\fM$. 
More explicitly, the elements of $\mathrm{Diff}_L(\fM)$ are diffeomorphisms $f: \mathrm{Dom}(f) \subset \fM \to f(\mathrm{Dom}(f)) \subset \fM$ of an open subset $\mathrm{Dom}(f) \subset \fM$ onto its image $f(\mathrm{Dom}(f)) \subset \fM$. We will denote the elements of $\mathrm{Diff}_L(\fM)$ as pairs $(f, \mathrm{Dom}(f))$. The local diffeomorphisms can be composed using the binary operation defined as
\begin{equation}
(f, \mathrm{Dom}(f)) \cdot (g, \mathrm{Dom}(g)) := (f \circ g, \mathrm{Dom}(f) \cap \mathrm{Dom}(g)),
\label{eq:composition}
\end{equation}
for all $(f, \mathrm{Dom}(f))$, $(g, \mathrm{Dom}(g)) \in \mathrm{Diff}_L(\fM)$.
It is easy to see that this operation is
associative and has $(\mathrm{id}, \fM)$, the identity map of $\fM$, as a (unique) two-sided identity element, which makes $\mathrm{Diff}_L(\fM)$ into a monoid (set with an associative operation which contains a two-sided identity element). 
Notice that only the elements in $\mathrm{Diff}(\fM) \subset \mathrm{Diff}_L(\fM)$ have an inverse since, in general, for any $(f, \mathrm{Dom}(f)) \in \mathrm{Diff}_L(\fM)$, we have that
\begin{align}
(f^{-1}, \mathrm{Dom}(f^{-1})) \cdot (f, \mathrm{Dom}(f)) &= (\mathrm{id}|_{\mathrm{Dom}(f)}, \mathrm{Dom}(f))\label{eq: left inverse}
\\
(f, \mathrm{Dom}(f)) \cdot (f^{-1}, \mathrm{Dom}(f^{-1})) \label{eq: right inverse} &= (\mathrm{id}|_{\mathrm{Dom}(f^{-1})}, \mathrm{Dom}(f^{-1})).
\end{align}
Consequently, the only way to obtain the identity element $(\mathrm{id}, \fM)$ out of the composition of $f$ with its inverse is having $\mathrm{Dom}(f) = \fM$. It follows from this argument that $\mathrm{Diff}(\fM)$ is the biggest subgroup contained in the monoid $\mathrm{Diff}_L(\fM)$ with respect to the composition law Eq.~\eqref{eq:composition}.

\subsubsection{Pseudogroups}
\begin{definition}[pseudogroup]
For a submonoid $A$ of $\mathrm{Diff}_L(\fM)$, if  for any $f: \mathrm{Dom}(f) \to f(\mathrm{Dom}(f)) \in A$ there exists another element $f^{-1}: f(\mathrm{Dom}(f)) \to \mathrm{Dom}(f)$ also in $A$ that satisfies the identities Eq.~\eqref{eq: left inverse} and Eq.~\eqref{eq: right inverse},
then $A$ is 
referred to as a pseudogroup of $\mathrm{Diff}_L(\fM)$.
\end{definition}
\begin{remark}
One of the important features of pseudogroups is that they have an associated orbit space. Indeed, if $A$ is a pseudogroup we define the orbit $A \cdot p$ under $A$ of any element $p \in \fM$ as the set $A \cdot p := \{f(p) \mid f \in A, \text{ such that } p \in \mathrm{Dom}(f)\}$. $A$ being a pseudogroup implies that the relation being in the same $A$-orbit is an equivalence relation and induces a partition of $\fM$ into $A$-orbits. 

\end{remark}

\subsubsection{Pseudogroups generated by arrow}

A significant number of integrable pseudogroups are generated by collections of arrows (see Ref.~\refcite{stefan1974accessible}). 

\begin{definition}[arrow]
An arrow is a differentiable mapping $\Phi: U \subset \mathbb{R} \times \fM \rightarrow \fM$ whose domain $U$ is an open subset of $\mathbb{R} \times \fM$ and that, additionally, satisfies:
\begin{enumerate}
    \item For every $t \in \mathbb{R}$, the map $\Phi_t: = \Phi(t, \cdot)$ is a local diffeomorphism of $\fM$ (possibly with empty domain).
    \item If the point $(t,p)$ belongs to the domain of $\Phi$, then so does $(s,p)$ for every $s \in [0, t]$. Moreover, $\Phi(0,p) = p$.
\end{enumerate} 
\end{definition}
An example of an arrow is the flow of an analytic vector field on $\fM$.
Let $E$ be a collection of arrows on $\fM$. We associate $E$ to a set $\mathcal{A}_{E} \subset \mathrm{Diff}_L(\fM)$ of local diffeomorphisms defined by $\mathcal{A}_{E} := \{\Phi_t \mid \Phi \in E, t\in \sR\}$, which at the same time, generates a pseudogroup
$$
A_{E} = (\mathrm{id}, \fM) \cup \bigcup_n \{\Phi_{1} \circ \dots \circ \Phi_{n} \mid n \in \mathbb{N} \text{ and for all } i=1,2,\ldots,n, \Phi_{i} \in \mathcal{A}_{E} \text{ or } (\Phi_{i})^{-1} \in \mathcal{A}_{E} \}.
$$

\subsection{Reachable Set and Orbit}
\begin{definition}[integral curve of a family of vector fields, Definition 5 in Chapter 1 of Ref.~\refcite{jurdjevic1997geometric}]
Let $\mathcal{F}$  be a family of analytic vector fields on an analytic manifold $\fM$.
A continuous curve $p(t)$ in $\fM$, defined on an interval $[0, T]$, is called an integral curve of $\mathcal{F}$ if there exist a partition $0 = t_0 < t_1 < \dots < t_k = T$ and vector fields $X_1, \dots, X_k$ in $\mathcal{F}$ such that the restriction of $p(t)$ to each open interval $(t_{i-1}, t_i)$ is differentiable, and $\frac{\mathrm{d}p(t)}{\mathrm{d}t} = X_i(p(t))$ for $i = 1, \dots, k$.

\end{definition}

\begin{definition}[reachable set, Definition 6 in Chapter 1 of Ref.~\refcite{jurdjevic1997geometric}]
 Let $\mathcal{F}$  be a family of analytic vector fields on an analytic manifold $\fM$.
\begin{enumerate}
    \item For each $T \geq 0$, and each $p_0$ in $\fM$, the set of points reachable from $p_0$ at time $T$, denoted by $R(p_0, T)$, is defined to be the set of the terminal points $p(T)$ of integral curves of $\mathcal{F}$ that originate at $p_0$.
    \item The union of $R(p_0, T)$, for $T \geq 0$, is called the set reachable from $p_0$. We will denote it by $R(p_0)$.
\end{enumerate}
    
\end{definition}

\begin{remark}[page 28 of Ref.~\refcite{jurdjevic1997geometric}]
The reachable sets admit further geometric descriptions through the following formalism. Assuming that the elements of $\mathcal{F}$ are all complete vector fields, then each element $X$ in $\mathcal{F}$ generates a one-parameter group of diffeomorphisms $\{e^{tX} \mid t \in \mathbb{R}^1\}$. Let $G(\mathcal{F})$ denote the subgroup of the group of diffeomorphisms in $\fM$ generated by the union of $\{e^{ tX} \mid t \in \mathbb{R}, X \in \mathcal{F}\}$. Each element $\Phi$ of $G(\mathcal{F})$ is a diffeomorphism of $\fM$ of the form
$$
\Phi = e^{ t_k X_k}\circ e^{ t_{k-1}X_{k-1}}\circ  \cdots \circ e^{ t_1 X_1}
$$
for some real numbers $t_1, \dots, t_k$ and vector fields $X_1, \dots, X_k$ in $\mathcal{F}$. $G(\mathcal{F})$ acts on $\fM$ in the obvious way and partitions $\fM$ into its orbits. Then the set reachable through $p_0$ at time $T$ consists of all points $\Phi(p_0)$ corresponding to elements $\Phi$ of $G(\mathcal{F})$ that can be expressed as $
\Phi = e^{ t_k X_k}\circ e^{ t_{k-1}X_{k-1}}\circ  \cdots \circ e^{ t_1 X_1}
$, with $t_1 \ge 0, \dots, t_k \ge 0, t_1 + \dots + t_k = T$, and $X_1, \dots, X_k$ in $\mathcal{F}$. The other reachable sets have analogous descriptions. In particular, $R(p_0)$ is equal to the orbit of the semigroup $S_\mathcal{F}$ through $p_0$, with $S_\mathcal{F}$ equal to the semigroup of all elements $\Phi$ in $G(\mathcal{F})$ of the form $
\Phi = e^{ t_k X_k}\circ e^{ t_{k-1}X_{k-1}}\circ  \cdots \circ e^{ t_1 X_1}
$, with $t_1 \ge 0, \dots, t_k \ge 0$, and $X_1, \dots, X_k$ in $\mathcal{F}$. The orbit of $S_\mathcal{F}$ through $p_0$, written as $S_\mathcal{F}(p_0)$, is equal to $\{\Phi(p_0) \mid \Phi \in S_\mathcal{F}\}$.

When some elements of $\mathcal{F}$ are not complete, then it becomes necessary to replace the corresponding groups of diffeomorphisms by local groups, and everything else remains the same. 
\end{remark}

\begin{definition}[orbit of family of vector fields]
Let $\mathcal{F}$ be a family of analytic vector fields on an analytic manifold $\fM$. Let $G$  denote the group (pseudogroup) of diffeomorphisms (local diffeomorphisms) generated by $\{e^{tX} \mid t \in \mathbb{R}, X \in \mathcal{F}\}$.
Then the orbit of $\mathcal{F}$ through $p$ is defined to be $\{\phi(p) \mid \phi \in G\}$. 
\end{definition}

\subsection{Lie Algebra}

\begin{definition}[Lie bracket, page 40 of Ref.~\refcite{jurdjevic1997geometric}]
Analytic vector fields act as derivations on the space of analytic functions. Moreover, If $X$ denotes an analytic vector field, and $h$ an analytic function on $\fM$, then $Xh$ will denote the function $p \mapsto X(p)(h)$.
For any analytic vector fields $X$ and $Y$, their Lie bracket $[X, Y]$ is defined by $[X, Y]h = Y(Xh) - X(Yh)$. 
\end{definition}
\begin{remark}[page 40 of Ref.~\refcite{jurdjevic1997geometric}]
If both $X$ and $Y$ are analytic vector fields on an analytic manifold $\fM$, then $[X, Y]$ is an analytic vector field. Let $p\in \fM$ and $(\phi,U)$ be a chart at $p$. The map $\phi:U\to \sR^n$ gives a  local coordinate $p\to (x_1(p),\ldots,x_n(p))$. In terms of the local coordinates, $[X, Y]$ is given by the following relations:
let $X(p) = \sum_{i=1}^n a_i(x_1, \dots, x_n)(\partial/\partial x_i)$, $Y(p) = \sum_{i=1}^n b_i(x_1, \dots, x_n)(\partial/\partial x_i)$, and $[X, Y](p) = \sum_{i=1}^n c_i(x_1, \dots, x_n)(\partial/\partial x_i)$. Then
\begin{equation}
c_i = \sum_{j=1}^n \left( \frac{\partial a_i}{\partial x_j}b_j - \frac{\partial b_i}{\partial x_j}a_j \right), \quad i = 1, 2, \dots, n.
\end{equation}
    
\end{remark}

\begin{definition}[Lie algebra of analytic vector fields, page 42 of Ref.~\refcite{jurdjevic1997geometric}]
Let $\mathfrak{X}^{\omega}(\fM)$ denote the space of all analytic vector fields on $\fM$. $\mathfrak{X}^{\omega}(\fM)$ is a real vector space under the pointwise addition of vectors
\begin{equation}
(\alpha X + \beta Y)(p) = \alpha X(p) + \beta Y(p) \quad \text{for all} \quad p \in \fM
\label{eq:linear}
\end{equation}
for each set of real numbers $\alpha$ and $\beta$ and vector fields $X$ and $Y$.
We shall regard $\mathfrak{X}^{\omega}(\fM)$ as an algebra, with the addition given by Eq.~\eqref{eq:linear} and with the product given by the Lie bracket.

\end{definition}

\section{Proofs}
\label{sec:proof}

\subsection{Proof of Proposition~\ref{prop:infinitesimal symmetry of neural network}}

\label{sec:proof of infinitesimal symmetry}
\begin{proof}
We use backpropagation to derive the gradients. For $l=1,\ldots,L$, define $\vz^{(l)}=\vW^{(l)}\va^{(l-1)}+\vb^{(l)}\in \sR^{n_l}$ and $\vdelta^{(l)}= \frac{\partial F}{\partial \vz^{(l)}} \in \mathbb{R}^{n_l}$.

Then $ \vdelta^{(L)} = \frac{\partial F}{\partial \vz^{(L)}} =1$,  and $\vdelta^{(l)} = \text{diag}(\sigma'(\vz^{(l)})) (\vW^{(l+1)})^\T \vdelta^{(l+1)}$ for $l=L-1,\ldots,1$. The partial derivatives for $\vW^{(l)}$ and $\vb^{(l)}$ are given by $\frac{\partial F}{\partial \vW^{(l)}} = \vdelta^{(l)} (\va^{(l-1)})^\T \in \mathbb{R}^{n_l \times n_{l-1}}$ and $\frac{\partial F}{\partial \vb^{(l)}} = \vdelta^{(l)} \in \mathbb{R}^{n_l}$.

(i)  Assume $\sigma(0)=0$. Let $\vLambda=(\vLambda^{(1)},\ldots,\vLambda^{(L-1)})$ be an arbitrary element in $G_{\mathrm{sign}}$. Let $\fM$ be the set of fixed points of $\vLambda$. Since $\vzero\in \fM$, $\fM\neq \emptyset$. Pick any $\vtheta=\left(\vW^{(l)}, \vb^{(l)}\right)_{l=1}^L\in \fM$ and any $\vx\in \sR^d$. For simplicity, we write $F(\vtheta)(\vx)$ as $F$. Since the action of $\vLambda$ is linear and orthogonal, by Remark~\ref{rmk:orthogonal}, to show that $\vLambda$ is an infinitesimal invariant map, it suffices to prove $\frac{\partial F}{\partial \vW^{(l)}}=\vLambda^{(l)}\frac{\partial F}{\partial \vW^{(l)}}\vLambda^{(l-1)}$ and $\frac{\partial F}{\partial \vb^{(l)}}=\vLambda^{(l)}\frac{\partial F}{\partial \vb^{(l)}}$ for $l=1,\ldots,L$.

    Since $\vtheta\in \fM$, $\vLambda^{(l)}\vW^{(l)}\vLambda^{(l-1)}=\vW^{(l)}$ and $\vLambda^{(l)}\vb^{(l)}=\vb^{(l)}$ for $l=1,\ldots,L$. For $l=1,\ldots,L$, define $I_l=\{i\in \{1,\ldots,n_l\}\mid \vLambda^{(l)}_{ii}=-1\}$ (note $I_0$ and $I_L$ are empty sets as $\vLambda^{(0)}=\vI, \vLambda^{(L)}=\vI$).
    Since $\vLambda^{(1)} \vW^{(1)}\vLambda^{(0)}=\vW^{(1)}$, the $j$-th row $\vW^{(1)}_j=\vzero$ for all $j\in I_1$. Similarly, $\vb^{(1)}_j=0$ for all $j\in I_1$. Thus $\vz^{(1)}_j=0$ for all $j\in I_1$. Since $\sigma(0)=0$, $\va^{(1)}_j=0$ for all $j\in I_1$.

    Next, we prove by induction that $\va^{(l)}_j=0$ for all $l\in\{1,\ldots,L\}$ and $j\in I_l$. Assume that $\va^{(l)}_j=0, \forall j\in I_l$ holds for some $l \in \{1,\ldots,L-1\}$.
    Since $\vLambda^{(l+1)}\vW^{(l+1)}\vLambda^{(l)}=\vW^{(l+1)}$, we know $\vW^{(l+1)}_{ij}=0$ if $i\in I_{l+1}$ and $j\notin I_{l}$, or if $i\notin I_{l+1}$ and $j\in I_{l}$.
    Since $\vLambda^{(l+1)}\vb^{(l+1)}=\vb^{(l+1)}$, we have $\vb^{(l+1)}_i=0, \forall i\in I_{l+1}$.
    Consider $\vz^{(l+1)}_i=\sum_{j=1}^{n_l}\vW^{(l+1)}_{ij} \va^{(l)}_j+\vb^{(l+1)}_i$ for any $i\in I_{l+1}$.
    If $j\notin I_l$, then $\vW^{(l+1)}_{ij}=0$.
       If $j \in I_l$, then $\va_j^{(l)}=0$ by the induction hypothesis.
    In both cases, $\vW^{(l+1)}_{ij} \va^{(l)}_j=0$. As $\vb^{(l+1)}_i=0$, we have $\vz^{(l+1)}_i=0, \forall i\in I_{l+1}$.  If $l=L-1$, then $\va^{(l+1)}=\vz^{(l+1)}.$ So we have $\va^{(l+1)}_i=0, \forall i\in I_{l+1}$. If $l<L-1$, then 
 $\va^{(l+1)}=\sigma(\vz^{(l+1)})$. Since $\sigma(0)=0$, it  holds that $\va^{(l+1)}_i=0, \forall i\in I_{l+1}$. In both cases, we have $\va^{(l+1)}_i=0, \forall i\in I_{l+1}$. 
 By mathematical induction, $\va^{(l)}_j=0$ for all $l\in \{1,\ldots,L\}, j\in I_l$.

    Next, we prove by backward induction that $\vdelta^{(l)}_i=0$ for all $l\in \{1,\ldots,L\}, i\in I_l$.
    When $l=L$, $I_L=\emptyset$, so the statement holds vacuously.
    Assume $\vdelta^{(l)}_i=0, \forall i\in I_{l}$ for some $l\in \{2,\ldots,L\}$. We have $\vdelta^{(l-1)}_i=\sigma'(\vz^{(l-1)}_i) \sum_{j=1}^{n_l} \vW^{(l)}_{ji}\vdelta^{(l)}_j$. For any $i\in I_{l-1}$, if $j\in I_l $, then $\vdelta^{(l)}_j=0$ by the induction hypothesis. If $j\notin I_l $, then $\vW^{(l)}_{ji}=0$ by the fixed point condition $\vLambda^{(l)}\vW^{(l)}\vLambda^{(l-1)}=\vW^{(l)}$, as $i \in I_{l-1}$.
    Thus, the sum is zero, and $\vdelta^{(l-1)}_i=0, \forall i\in I_{l-1}$. By induction, $\vdelta^{(l)}_i=0, \forall l\in \{1,\ldots,L\}, i\in I_l$.

    For any $l=1,\ldots,L$, $\frac{\partial F}{\partial \vW^{(l)}} = \vdelta^{(l)} (\va^{(l-1)})^\T$. Since $\vdelta^{(l)}_i=0, \forall i \in I_l$, the $i$-th row of $\frac{\partial F}{\partial \vW^{(l)}}$ is zero. Since $\va^{(l-1)}_j=0, \forall j \in I_{l-1}$, the $j$-th column is zero. This implies $\frac{\partial F}{\partial \vW^{(l)}}=\vLambda^{(l)}\frac{\partial F}{\partial \vW^{(l)}}\vLambda^{(l-1)}$.
    For any $l=1,\ldots,L$, $\frac{\partial F}{\partial \vb^{(l)}} = \vdelta^{(l)}$. Since $\vdelta^{(l)}_i=0, \forall i\in I_l$, this implies $\frac{\partial F}{\partial \vb^{(l)}}=\vLambda^{(l)}\frac{\partial F}{\partial \vb^{(l)}}$.
    Therefore, $\vLambda$ is an infinitesimal invariant map.

(ii)  Assume $\sigma'(0)=0$.  Let $\vLambda=(\vLambda^{(1)},\ldots,\vLambda^{(L-1)})$ be an arbitrary element in $G_{\mathrm{sign}}'$.  Let $\fM$ be the set of fixed points of $\vLambda$. Since $\vzero\in \fM$, $\fM\neq \emptyset$. Pick any $\vtheta=\left(\vW^{(l)}, \vb^{(l)}\right)_{l=1}^L\in \fM$. Similar to (1), to show $\vLambda$ is an infinitesimal invariant map, we only need to prove $\frac{\partial F}{\partial \vW^{(l)}}=\vLambda^{(l)}\frac{\partial F}{\partial \vW^{(l)}}$ and $\frac{\partial F}{\partial \vb^{(l)}}=\vLambda^{(l)}\frac{\partial F}{\partial \vb^{(l)}}$ for $l=1,\ldots,(L-1)$.

For $l=1,\ldots,(L-1)$, define $I_l=\{i\in \{1,\ldots,n_l\}\mid \vLambda^{(l)}_{ii}=-1\}$. Since $\vtheta\in \fM$, $\vLambda^{(l)}\vW^{(l)}=\vW^{(l)}$ and $\vLambda^{(l)}\vb^{(l)}=\vb^{(l)}$ for $l=1,\ldots,(L-1)$.
    This implies that for any $l=1,\ldots,(L-1)$ and $j\in I_l$, the $j$-th row $\vW^{(l)}_j=\vzero$ and $\vb^{(l)}_j=0$.
    Thus, $\vz^{(l)}_j = (\vW^{(l)}\va^{(l-1)})_j + \vb^{(l)}_j = \vzero \cdot \va^{(l-1)} + 0 = 0$ for all $j\in I_l$.
         For $l<L$, $\vdelta^{(l)} = \text{diag}(\sigma'(\vz^{(l)})) (\vW^{(l+1)})^\T \vdelta^{(l+1)}$. Since $\vz^{(l)}_j=0$ for $j \in I_l$, the $j$-th diagonal entry $\sigma'(\vz^{(l)}_j) = \sigma'(0) = 0$, which implies $\vdelta^{(l)}_j=0$.
    Thus, for all $l=1,\ldots,(L-1)$, we have $\vdelta^{(l)}_j=0$ for all $j\in I_l$.

    Since $\frac{\partial F}{\partial \vW^{(l)}} = \vdelta^{(l)} (\va^{(l-1)})^\T $ and $\frac{\partial F}{\partial \vb^{(l)}} = \vdelta^{(l)}$, the $j$-th row of both $\frac{\partial F}{\partial \vW^{(l)}}$ and $\frac{\partial F}{\partial \vb^{(l)}}$ is zero for all $j\in I_l$.
    This directly implies $\frac{\partial F}{\partial \vW^{(l)}}=\vLambda^{(l)}\frac{\partial F}{\partial \vW^{(l)}}$ and $\frac{\partial F}{\partial \vb^{(l)}}=\vLambda^{(l)}\frac{\partial F}{\partial \vb^{(l)}}$ for all $l=1,\ldots,(L-1)$. Thus, $\vLambda$ is an infinitesimal invariant map.
\end{proof}

\subsection{Proof of Proposition~\ref{example:sim induced by permutation symmetry group}}

\label{sec:proof of sim induced by permutation symmetry}
\begin{proof}
For two-layer neural networks of width $m$, $G_{\mathrm{per}}$ is simply the group of $m\times m$ permutation matrices, denoted by $S_m$ . Since it is isomorphic to the symmetric group of degree $m$, we sightly abuse the notation by denoting $S_m$ to be the symmetric group of degree $m$ in the proof. 
The action of $S_m$ on the parameter space is given by:
$$\pi: (a_i,\vw_i)_{i=1}^m \mapsto (a_{\pi^{-1}(i)}, \vw_{\pi^{-1}(i)})_{i=1}^m ,\forall \pi\in S_m.$$
For any parameter vector $\vtheta = (a_i, \vw_i)_{i=1}^m \in \mathbb{R}^{(d+1)m}$, we define an equivalence relation $\sim$ on the set of indices $\{1, \ldots, m\}$ such that $i \sim j$ if and only if $(a_i, \vw_i) = (a_j, \vw_j)$. This relation induces a partition of $\{1, \ldots, m\}$, which we denote by $\mathcal{P}_{\vtheta} = \{B_1, \ldots, B_s\}$.

Fix any $\vtheta=(a_i,\vw_i)_{i=1}^m \in \sR^{(d+1)m}$.
Let $S(\vtheta)$ be the stabilizer subgroup of $\vtheta$ in $S_m$. An element $\pi \in S_m$ belongs to $S(\vtheta)$ if and only if $\pi(\vtheta) = \vtheta$, i.e. $(a_{\pi^{-1}(i)}, \vw_{\pi^{-1}(i)}) = (a_i, \vw_i),\forall i \in \{1, \ldots, m\}$.
By definition of $\fP_{\vtheta}$,
this holds if and only if for each $i$, the indices $i$ and $\pi^{-1}(i)$ belong to the same block in the partition $\mathcal{P}_{\vtheta}$. This is true if and only if $\pi$ permutes the indices within each block $B_l$ for $l=1, \ldots, s$.
Consequently, 
\begin{equation}
S(\vtheta) = S_{|B_1|} \times S_{|B_2|} \times \cdots \times S_{|B_s|},
\label{eq:stheta}
\end{equation}
where $S_{|B_l|}$ is the symmetric group on the set $B_l$.
One can readily  verify that \\
(i): If $\vtheta\in \fM_{\fP_{\vtheta_0}}$, then $\fP_{\vtheta}=\fP_{\vtheta_0}$. Thus, $S(\vtheta)=S(\vtheta_0)$.\\
(ii): If $\vtheta\notin \fM_{\fP_{\vtheta_0}}$, then $\fP_{\vtheta}\neq\fP_{\vtheta_0}$. Thus, $S(\vtheta)\neq S(\vtheta_0)$.\\
Then by definition of the equivalence class, we have $[\vtheta_0]=\fM_{\fP_{\vtheta_0}}$. Since  $\vtheta_0$ is arbitrary, 
 $\{ \mathcal{M}_{\mathcal{P}} \mid \mathcal{P} \in \mathfrak{P}_m \}$ is the invariant partition induced by $G_{\mathrm{per}}$.
\end{proof}

\subsection{Proof of Proposition~\ref{example:sim induced by combined symmetry of tanh network}}
\label{sec:proof of SIM induced by the combined}
\begin{proof}
For two-layer neural networks, $G_{\mathrm{combine}}$ is simply
$S_2^m \rtimes S_m$ defined in Definition~\ref{def:symmetry group}. With a slight abuse of notation, $S_m$ and $S_2^m$ are denoted to be the symmetric group of order $m$ and $\{-1,1\}^m$ ($m$ products of the group $\{-1,1\}$), respectively.
Then an element of $G$ is a pair $(\vdelta, \pi)$ where $\vdelta = (\delta_1, \ldots, \delta_m) \in \{-1,1\}^m$ and $\pi \in S_m$. The action of the pair $(\vdelta,\pi)$
 is given by:
$$
(\vdelta, \pi): (a_i, \vw_i)_{i=1}^m \mapsto \left( \delta_i a_{\pi^{-1}(i)}, \delta_i \vw_{\pi^{-1}(i)} \right)_{i=1}^m.
$$

First, we note that for any $\vtheta \in \mathbb{R}^{(d+1)m}$, we can find a partition $\mathcal{P} \in \mathfrak{P}_m$ and a sign vector $\vgamma \in \{-1,1\}^m$ such that $\vtheta \in \mathcal{M}_{\mathcal{P}, \vgamma}$. Thus, the collection of sets $\{ \mathcal{M}_{\mathcal{P}, \vgamma} \}$ covers the entire parameter space. Therefore, to prove that this collection is the invariant partition of $S_2^m\rtimes S_m$, we only need to show that for any $\fP\in \mathfrak{P}_m, \vgamma\in \{-1,1\}^m$, any  $\vtheta_0 \in \mathcal{M}_{\mathcal{P}, \vgamma}$, the identity $[\vtheta_0] = \mathcal{M}_{\mathcal{P}, \vgamma}$ holds. Now fix any $\mathcal{P}=\{B_1,\ldots,B_s\}\in \mathfrak{P}_m,$ any $ \vgamma=(\gamma_1,\ldots,\gamma_m)\in \{-1,1\}^m$, and any $\vtheta_0 \in \mathcal{M}_{\mathcal{P}, \vgamma}$.
\\
\textbf{Step 1:  Prove $\fM_{\fP,\vgamma} \subset [\vtheta_0]$.} 
Let $\vtheta=(a_i,\vw_i)_{i=1}^m  $ be an arbitrary element in $\mathcal{M}_{\mathcal{P}, \vgamma}$. Denote $G=S_2^m\rtimes S_m$.
An element $(\vdelta, \pi) \in G$ is in  $S(\vtheta)$ if and only if $(\vdelta, \pi) \cdot \vtheta = \vtheta$, which means:
\begin{equation} \label{eq:stabilizer_condition}
(a_i, \vw_i) = \delta_i (a_{\pi^{-1}(i)}, \vw_{\pi^{-1}(i)}) \quad \text{for all } i \in \{1, \ldots, m\}.
\end{equation}
If $(\vdelta, \pi) \in S(\vtheta)$, then $(\vdelta, \pi)$ must satisfy the following conditions:
\begin{enumerate}
    \item From the third condition of $\mathcal{M}_{\mathcal{P}, \vgamma}$, we have $(a_i, \vw_i) \ne \pm(a_j, \vw_j)$ if $i$ and $j$ are in different blocks of the partition $\mathcal{P}$. Equation~\eqref{eq:stabilizer_condition} can only hold if for every block $B_l \in \mathcal{P}$, the permutation $\pi$ maps $B_l$ to itself, i.e., $\pi(B_l) = B_l$.

    \item For any $i \in B_1$, we have $(a_i, \vw_i) = \vzero$. Since $\pi(B_1)=B_1$, $\pi^{-1}(i)$ is also in $B_1$, so $(a_{\pi^{-1}(i)}, \vw_{\pi^{-1}(i)}) = \vzero$. The condition becomes $\vzero = \delta_i \vzero$, which holds for any $\delta_i \in \{-1,1\}$.

    \item For any $i \in B_l$ with $l \in \{2, \ldots, s\}$, we have $(a_i, \vw_i) \ne \vzero$. From the second condition of $\mathcal{M}_{\mathcal{P}, \vgamma}$, we know that $\gamma_i (a_i, \vw_i) = \gamma_j (a_j, \vw_j)$ for any $i, j \in B_l$. Applying this to the pair $i, \pi^{-1}(i) \in B_l$, we get $\gamma_i (a_i, \vw_i) = \gamma_{\pi^{-1}(i)} (a_{\pi^{-1}(i)}, \vw_{\pi^{-1}(i)})$. Substituting this into the stabilizer condition Eq.~\eqref{eq:stabilizer_condition}, we find:
    $$ \gamma_i (a_i, \vw_i) = \gamma_{\pi^{-1}(i)} \left( \frac{1}{\delta_i} (a_i, \vw_i) \right) \implies \delta_i = \frac{\gamma_{\pi^{-1}(i)}}{\gamma_i}. $$
Since $\gamma_k \in \{-1,1\}$, this is equivalent to $\delta_i = \gamma_i \gamma_{\pi^{-1}(i)}$.
\end{enumerate}
\noindent Denote $H$ to be the set of all pairs $(\vdelta, \pi) \in G$ such that:
\begin{enumerate}
    \item $\pi(B_l) = B_l$ for all $l \in \{1, \ldots, s\}$.
    \item $\delta_i \in \{-1,1\}$ is arbitrary for $i \in B_1$.
    \item $\delta_i = \gamma_i \gamma_{\pi^{-1}(i)}$ for all $i \in B_l$ where $l \in \{2, \ldots, s\}$.
\end{enumerate}
By previous argument, $S(\vtheta)\subset H$. 
Conversely, it is easy to verify that Eq.~\eqref{eq:stabilizer_condition} holds whenever $(\vdelta,\pi)\in H$. Therefore $H\subset S(\vtheta)$. By both inclusions, $H=S(\vtheta)$. 
Since $H$ depends only on the partition $\mathcal{P}$ and the sign vector $\vgamma$, all elements of $\mathcal{M}_{\mathcal{P}, \vgamma}$ have the same stabilizer subgroup, and thus $\mathcal{M}_{\mathcal{P}, \vgamma}\subset [\vtheta_0]$. 
\\
\textbf{Step 2:  Prove $[\vtheta_0]\subset\fM_{\fP,\vgamma}$.} Let $\vtheta'=(a_i',\vw_i')_{i=1}^m $ be an arbitrary element in $ [\vtheta_0]$. We will show that $\vtheta' \in \mathcal{M}_{\mathcal{P}, \vgamma}$.
Since $S(\vtheta') = S(\vtheta_0)$, the following conditions hold:
\begin{enumerate}
    \item For any $i \in B_1$, the element $(\vdelta, \text{id})$ where $\delta_i = -1$ and $\delta_j=1$ for $j \ne i$ is in $S(\vtheta_0)$.
    Since $S(\vtheta')=S(\vtheta_0)$, $(\vdelta, \text{id})\in S(\vtheta')$. Therefore $(\vdelta, \text{id})\cdot \vtheta'=\vtheta'$, which implies $(a'_i, \vw'_i) = -(a'_i, \vw'_i)$, so $(a'_i, \vw'_i) = \vzero$. This holds for all $i \in B_1$.
    \item For any $l \in \{2, \ldots, s\}$ and any $i,j \in B_l$, let $\pi_{ij}$ be the transposition of $i$ and $j$. The element $(\vdelta, \pi_{ij})$ with $\delta_k = \gamma_k \gamma_{\pi_{ij}^{-1}(k)}$ is in $S(\vtheta_0)$. Applying it to $\vtheta'$ at index $i$ gives $(a'_i, \vw'_i) = \delta_i (a'_j, \vw'_j) = (\gamma_i \gamma_j)(a'_j, \vw'_j)$, which implies $\gamma_i (a'_i, \vw'_i) = \gamma_j(a'_j, \vw'_j)$.
    \item If $\vtheta'$ violated the third condition, i.e., if $(a'_i, \vw'_i) = \pm (a'_j, \vw'_j)$ for some $i \in B_l, j \in B_{l'}$ with $l \ne l'$, then $S(\vtheta')$ would contain elements $(\vdelta, \pi)$ where $\pi(i)=j$. Such elements are not in $S(\vtheta_0)$, contradicting $S(\vtheta_0)=S(\vtheta')$.
\end{enumerate}
Thus, $\vtheta'$ must satisfy all three conditions defining $\mathcal{M}_{\mathcal{P}, \vgamma}$. So
$\vtheta'\in \mathcal{M}_{\mathcal{P}, \vgamma}$. Thus $[\vtheta_0]\subset \mathcal{M}_{\mathcal{P}, \vgamma}$.

\end{proof}

\subsection{Proof of Corollary~\ref{cor:all rank of lie closure}}
\label{sec:proof of all rank of lie closure}
\begin{proof}
We claim that, for two-layer neural networks, the calculation of Lie algebra is neuron-wise. We begin with the following definitions.  
Define the  network of width $k$ as  ${F_k}(\vtheta)(\vx)=\sum_{i=1}^k a_i\sigma(\vw_i^\T\vx),\vx\in \sR^d, \vtheta=(a_i,\vw_i)\in \sR^{(d+1)k}$, and define $\fF_k=\{\nabla_{\vtheta}F_k(\cdot)(\vx)\mid \vx\in \sR^d\}$. Denote $\mathrm{Lie}(\fF_k)$ to be 
 the Lie algebra  generated by $\fF_k$. 
 Define $\fF_k'=\{(X(a_i,\vw_i))_{i=1}^k\mid X\in \mathrm{Lie}(\fF_1)\}$, which is a family of vector fields on $\sR^{(d+1)k}$. The specific model considered is ${F}(\vtheta)(\vx)=\sum_{i=1}^m a_i\sigma(\vw_i^\T\vx),\vx\in \sR^d, \vtheta=(a_i,\vw_i)\in \sR^{(d+1)m}$ for fixed $m\in \sN^+$.  For notional simplicity, 
 we omit the subscript $m$,
 using $F,\mathrm{Lie}(\fF), \fF'$ to denote $F_m,\mathrm{Lie}(\fF_m), \fF_m'$, respectively.

For any positive integer \( k \), and any \( X_1, \ldots, X_k \in \fF_1 \), define \( Y_j = (X_j(a_i, \vw_i))_{i=1}^m \) for \( j = 1, \ldots, k \). Define the nested Lie brackets \( X = [X_1, [X_2, [\cdots [X_{k-1}, X_k] \cdots ]]] \) and \( Y = [Y_1, [Y_2, [\cdots [Y_{k-1}, Y_k] \cdots ]]] \). It is straightforward to verify by induction that \( Y = (X(a_i, \vw_i))_{i=1}^m \).

Let \( \mathfrak{g}^{(k)} \) and \( \mathfrak{g}_1^{(k)} \) denote the \( k \)-th terms in the lower central series of \( \fF \) and \( \fF_1 \), respectively. The lower central series \( \mathfrak{g}^{(1)}, \mathfrak{g}^{(2)}, \ldots \) of a family of vector fields \( \tilde{\fF} \) is defined recursively by \( \mathfrak{g}^{(1)} = \tilde{\fF} \) and \( \mathfrak{g}^{(k+1)} = [\mathfrak{g}, \mathfrak{g}^{(k)}] \), where the bracket denotes the Lie bracket. 
From the identity \( Y = (X(a_i, \vw_i))_{i=1}^m \), and the fact that \( \fF = \{ (X(a_i, \vw_i))_{i=1}^m \mid X \in \fF_1 \} \), it follows that \( \mathfrak{g}^{(k)} = \{ (X(a_i, \vw_i))_{i=1}^m \mid X \in \mathfrak{g}_1^{(k)} \} \) for all \( k \in \sN^+ \). Consequently, we conclude that \( \fF' = \mathrm{Lie}(\fF) \).

(i) Assume the activation function to be a generic activation.
Fix any $\vtheta=(a_i,\vw_i)_{i=1}^m\in \sR^{(d+1)m}$. By Proposition~\ref{example:sim induced by permutation symmetry group}, $[\vtheta]=\fM_\fP$ for some partition $\fP=\{B_1,\ldots,B_s\}$ of $\{1,\ldots,m\}$. By definition of $\fM_{\fP}$, we have
$\dim([\vtheta])=\dim(\fM_{\fP})=(d+1)s$. Next we calculate $\mathrm{dim(Lie_{\vtheta}(\fF)})$. Since $\mathrm{Lie}(\fF)=\fF'$, we have $\mathrm{Lie_{\vtheta}(\fF)}=\fF'|_{\vtheta}$. For each $p=1,\ldots,s$, select $k_p\in B_p$. We then define $\vtheta'=(a_{k_p}, \vw_{k_p})_{p=1}^s$. Similarly, $\mathrm{Lie_{\vtheta'}(\fF_s)}=\fF_s'|_{\vtheta'}$.
Define a linear map $P:\fF'|_{\vtheta}\to \sR^{(d+1)s}$ by: $(a_i^*,\vw_i^*)_{i=1}^m\mapsto (a^*_{k_p}, \vw^*_{k_p})_{p=1}^s,\forall (a_i^*,\vw_i^*)_{i=1}^m\in \fF'|_{\vtheta}$. 
By definition of $\fF_s'|_{\vtheta'}$,  $P$ is a surjection onto $\fF_s'|_{\vtheta'}$ . It implies that $\mathrm{dim(Lie}_{\vtheta}(\fF))=\dim(\fF'|_{\vtheta})\geq \dim(\fF_s'|_{\vtheta'})=\mathrm{dim(Lie}_{\vtheta'}(\fF_s))$. Since $\vtheta\in \fM_\fP$, it follows that $(a_{k_p}, \vw_{k_p})\neq (a_{k_q}, \vw_{k_q})$ for any $p,q\in \{1,\ldots,s\}$ and $p\neq q$.
By Corollary~\ref{cor:pertube to non-degenerate},  $\mathrm{dim(Lie_{\vtheta'}(\fF_s)})=(d+1)s$. Therefore $\mathrm{dim(Lie_{\vtheta}(\fF)})\geq\mathrm{dim(Lie_{\vtheta'}(\fF_s)})= (d+1)s$. By Theorem~\ref{thm:Hermann-Nagano},  
the tangent space of $O_{\fF}(\vtheta)$ at $\vtheta$ is $\mathrm{Lie}_{\vtheta}(\fF)$.  By Lemma~\ref{lemma: invariant partition}, $[\vtheta]$ is a SIM. 
So
$O_{\fF}(\vtheta)\subset [\vtheta]$. Take this inclusion to tangent space gives $\mathrm{dim(Lie_{\vtheta}(\fF)})\leq (d+1)s$. Thus, $\mathrm{dim(Lie_{\vtheta}(\fF)})= (d+1)s$.

(ii) We now consider the case where the activation function is a generic odd function and the symmetry group is the combined orthogonal group. The proof is analogous to that of (i); hence, we omit several details for brevity.
By  Proposition~\ref{example:sim induced by combined symmetry of tanh network}, we have $[\vtheta]=\fM_{\fP,\vgamma}$ for some $\fP=\{B_1,\ldots,B_s\}$ and $\vgamma\in \{-1,1\}^m$. Then dimension is therefore $\mathrm{dim}([\vtheta])=(d+1)(s-1)$. A similar argument establishes that 
$\mathrm{Lie}_{\vtheta}(\fF)\geq \mathrm{Lie}_{\vtheta'}(\fF_{s-1})$, where $\vtheta'=(a_{k_p},\vw_{k_p})_{p=2}^s$ for $k_p\in B_p$ with $p=2,\ldots,s$. Subsequently, by Corollary~\ref{cor:pertube to non-degenerate}, it follows that $\mathrm{dim(Lie_{\vtheta'}(\fF)})=(d+1)(s-1)$. Then 
some straightforward
reasoning leads to $\mathrm{dim(Lie_{\vtheta}(\fF)})=\mathrm{dim}([\vtheta])= (d+1)(s-1)$.

\end{proof}

\section{Experiment Details}
\subsection{Figure~\ref{fig:invariance}}
\label{sec:details of invariance}
In the experiment,
we consider a two-layer neural network with a $\tanh$ activation function. The network has an input dimension of 20, a hidden layer with 100 neurons, and a single output unit. We initialize the network such that it lies on a SIM by setting the initialization of the first two hidden neurons to be identical, i.e., $a_1(0)=a_2(0),\vw_1(0)=\vw_2(0)$. 

The model is then trained using  SGD with learning rate $0.1$. To rigorously test the data-independence, we employ a training scheme where, at each step, a new mini-batch of data (inputs and targets) is drawn from a standard normal distribution. 
We conduct three independent training runs, each with a different loss function:

\begin{enumerate}
    \item Mean Squared Error (MSE) Loss:
    $
    l_{\text{MSE}}(\hat{y}, y) = (\hat{y} - y)^2.
    $
    
    \item Binary Cross-Entropy (BCE) Loss with Logits:
    \[
    l_{\text{BCE}}(\hat{y}, y) = - y \cdot \log(\phi(\hat{y})) - (1 - y) \cdot \log(1 - \phi(\hat{y})),
    \]
    where
    $
    \phi(\hat{y}) = \frac{1}{1 + e^{-\hat{y}}}
    $
    is the sigmoid activation function.
    
    \item Log-Cosh Loss:
$l_{\text{LogCosh}}(\hat{y}, y) = \log\left(\cosh(\hat{y} - y)\right).$
\end{enumerate}

We track the Euclidean distance \(\|\boldsymbol{\theta}_1(t) - \boldsymbol{\theta}_2(t)\|_2\) between the parameters of two initially identical neurons during training. As shown in Figure~\ref{fig:invariance}, this distance remains below \(10^{-50}\) over all 2000 training steps, effectively zero at every iteration. This invariance holds across all loss functions tested, and persists even when training data is resampled at each step. The experimental results support that the invariance of SIMs  is independent of both the data and the loss function. Furthermore, the results suggest that this invariance is also relevant to practical optimization methods such as SGD.

\subsection{Figure~\ref{fig:normal_dynamics}}

\label{sec:details of normal dynamics}

The model under consideration is a two-layer neural network with input dimension 2 and hidden layer width 2. The parameters of the first hidden neuron are initialized by sampling from a Gaussian distribution with mean zero and standard deviation 0.1. The second hidden neuron is initialized by adding an independent perturbation drawn from a Gaussian distribution with mean zero and standard deviation \(\varepsilon\) to the parameters of the first neuron. We consider \(\varepsilon = 0\) (corresponding to exact initialization on the structural invariant manifold, SIM), as well as \(\varepsilon = 10^{-1}, 10^{-3}, 10^{-5}\), to systematically investigate the effect of small deviations from the manifold.

The target function is defined by a width-1 two-layer neural network with all weights set to 1. The training dataset consists of four input vectors \(\mathbf{x}\), sampled independently from a standard Gaussian distribution. The corresponding labels \(y\) are generated by passing \(\mathbf{x}\) through the target network.

The model is trained using gradient descent with a learning rate of 0.3 to accelerate convergence. We define the neuron distance as the Euclidean norm between the parameter vectors of the two hidden neurons, i.e., \(\|\boldsymbol{\theta}_1 - \boldsymbol{\theta}_2\|_2\).

To evaluate test performance, we generate 100 input vectors \(\mathbf{x}\) sampled from a standard normal distribution and compute the test loss as the \(\ell_2\) distance between the outputs of the trained network and those of the target function.

\section*{Acknowledgment}

We are grateful to Leyang Zhang for his valuable insights and suggestions for refining this paper\footnote{Email: leyangz\_hawk@outlook.com}. This work is sponsored by the National Key R\&D Program of China Grant No. 2022YFA1008200 (Y.Z., T.L.), Natural Science Foundation of China No. 12571567 (Y.Z.), Natural Science Foundation of Shanghai No. 25ZR1402280 (Y.Z.), Shanghai Institute for Mathematics and Interdisciplinary Sciences Grant No. SIMIS-ID-2025-ST (T.L.).

\bibliography{sample1}
\bibliographystyle{ws-m3as}

\end{document}